\documentclass[a4paper, 12pt]{amsart}
\usepackage{hyperref,upref,amssymb,tikz}

\numberwithin{equation}{section}

\newtheorem{thm}{Theorem}[section]
\newtheorem{lemma}[thm]{Lemma}
\newtheorem{prop}[thm]{Proposition}
\newtheorem{cor}[thm]{Corollary}
\theoremstyle{definition}
\newtheorem{defn}[thm]{Definition}

\theoremstyle{remark}
\newtheorem{eg}[thm]{Example}
\newtheorem{rmk}[thm]{Remark}

\newcommand{\supp}{\operatorname{supp}}
\newcommand{\lsp}{\operatorname{span}}

\newcommand{\id}{\operatorname{id}}

\newcommand{\Ind}{\mathrm{Ind}}

\title{Twisted Topological Graph Algebras}

\date{30 April 2014}

\author{Hui Li}
\email{hl338@uowmail.edu.au}

\address{School of Mathematics and Applied Statistics \\
Building 39C\\
University of Wollongong\\
Wollongong NSW 2522\\
AUSTRALIA}

\thanks{This research was supported by the Research Career Launch Scholarship at University of Wollongong}

\begin{document}

\begin{abstract}
We define the notion of a twisted topological graph algebra associated to a topological graph and a $1$-cocycle on its edge set. We prove a stronger version of a Vasselli's result. We expand Katsura's results to study twisted topological graph algebras. We prove a version of the Cuntz-Krieger uniqueness theorem, describe the gauge-invariant ideal structure. We find that a twisted topological graph algebra is simple if and only if the corresponding untwisted one is simple.
\end{abstract}

\maketitle

\section{Introduction}

Since the foundational work in \cite{CuntzKrieger:IM80, EnomotoWatatani:MJ80, KumjianPaskEtAl:PJM98, KumjianPaskEtAl:JFA97}, directed graph algebras have been studied very extensively over the last twenty years. Many properties like the ideal structure and the $K$-theory of graph algebras can be read off from the underlying graph (see \cite{Raeburn:Graphalgebras05} for a detailed introduction of graph algebras). Graph algebras have provided very illustrative examples for the development of the classification theory of $C^*$-algebras. For example, it was shown in \cite{DrinenTomforde:RMJM05, KumjianPaskEtAl:PJM98} that simple graph algebras are either AF or purely infinite so they are classifiable up to isomorphism by the K-theory. For two directed graphs whose graph algebras are simple unital, S{\o}rensen recently in \cite{MR3082546} showed how to decide exactly when these two graphs determine stable isomorphic graph algebras.

Various versions of continuous graph algebras have been studied by many authors. For example, Deaconu in \cite{MR1233967} investigated the groupoid $C^*$-algebra associated with a local homeomorphism. In \cite{Katsura:TAMS04}, Katsura defined the concept of a topological graph and associated a topological graph algebra to each topological graph by modifying Pimsner's construction in \cite{Pimsner:FIC97} of a $C^*$-algebra from a $C^*$-correspondence. Topological graph algebras include all graph algebras, all homeomorphism algebras, all AF-algebras, and many other examples. One of the famous results about topological graph algebras \cite[Theorem~C]{Katsura:JFA08} states that every Kirchberg algebra is isomorphic to a topological graph algebra. Muhly and Tomforde in \cite{MuhlyTomforde:IJM05} have subsequently considered the $C^*$-algebra associated to a topological quiver which is a further generalization of a topological graph algebra.

There are many interesting examples of twisted $C^*$-algebras, which incorporate suitable cohomological data into existing constructions of $C^*$-algebras. People are interested in twisted $C^*$-algebras because properties of twisted $C^*$-algebras frequently exhibit strong connections with the twisting cohomology data. Examples of twisted $C^*$-algebras include: twisted crossed products \cite{PackerRaeburn:MPCPS89}, twisted groupoid $C^*$-algebras obtained from a local homeomorphism in \cite{DeaconuKumjianEtAl:JOT01}, and twisted $k$-graph algebras \cite{MR2948223,twisted k graph algebra}. The survey paper \cite{RaeburnSimsEtAl:00} provides lots of interesting examples and gives a detailed motivation for studying twisted $C^*$-algebras.

In this paper we invoke sheaf cohomology theory from \cite{RaeburnWilliams:Moritaequivalenceand98} to generalize Katsura's graph correspondences in \cite{Katsura:TAMS04} to twisted ones, which will be done in Section~3. We mainly study the Cuntz-Pimsner algebra of a twisted graph correspondence, which we regard as the twisted topological graph algebra. In Section~4 we prove a stronger version of a result of Vasselli in \cite{MR1966825}, which provides a large class of twisted topological graph algebras that are not isomorphic to the ordinary topological graph algebras of the same graphs. In Section~5, we prove a series of technical results generalizing parts of Katsura's work in \cite{Katsura:TAMS04, Katsura:IJM06, Katsura:ETDS06}. These results allow us to prove versions of the fundamental structure theorems for twisted topological graph algebras. In Section~6, we prove a version of the Cuntz-Krieger uniqueness theorem for twisted topological graph algebras. In Section~7, we establish a complete characterization of the gauge-invariant closed two-sided ideals of a twisted topological graph algebra. In Section~8, we give some conditions which are equivalent to the simplicity of the twisted topological graph algebra. In particular, we prove that the twisted topological graph algebra of a topological graph is simple if and only if the ordinary topological graph algebra is simple.

\section{Preliminaries}

Throughout this paper, we adopt the following conventions. Let $T$ be a locally compact Hausdorff space, and let $U$ be an open subset of $T$. For $f \in C_0(U)$, and $g \in C_b(U)$, define a function $f \times g \in C_0(U)$ by $f \times g(t) := f(t)g(t)$ if $t \in U$. For a closed two-sided ideal $J$ of a $C^*$-algebra $A$, define a closed two-sided ideal of $A$ by $J^{\perp}:= \{a \in A:ab=0\text{ for all }b \in J\}$. For a right Hilbert $A$-module $X$, define a closed right $A$-submodule $X_J:=\overline{\lsp}\{x \cdot j:x \in X,j \in J\}$.

In this section, we recall some background about $C^*$-correspondences and topological graphs which will be used throughout this paper. 

First of all, let us recap the material about $C^*$-correspondences from \cite{FowlerRaeburn:IUMJ99, Katsura:JFA04, Pimsner:FIC97}.

Let $A$ be a $C^*$-algebra, let $X$ be a right Hilbert $A$-module, and let $\phi:A \to \mathcal{L}(X)$ be a homomorphism. Then the pair $(X,\phi)$ is called a \emph{$C^*$-correspondence} over $A$. A right Hilbert $A$-module $X$ is a $C^*$-correspondence if and only if there is a left action $\cdot: A \times X \to X$ such that $\langle a^* \cdot y,x \rangle_A=\langle y,a\cdot x\rangle_A$. In the rest of this paper we will refer to as $X$ is a $C^*$-correspondence over $A$.

Fix a $C^*$-correspondence $X$ over a $C^*$-algebra $A$. A pair $(\psi,\pi)$ consisting of a linear map $\psi:X \to B$ and a homomorphism $\pi:A \to B$ is called a \emph{Toeplitz
representation} of $X$ in a $C^*$-algebra $B$ if $\psi(a \cdot x)=\pi(a)\psi(x)$ and $\psi(x)^*\psi(y)=\pi(\langle x,y \rangle_A)$. Define $C^*(\psi,\pi)$ to be the $C^*$-subalgebra of $B$ generated by the images of $\psi$ and $\pi$. Proposition~1.3 of \cite{FowlerRaeburn:IUMJ99} shows that there is a $C^*$-algebra $\mathcal{T}_X$ generated by an injective universal Toeplitz representation $(i_X,i_A)$ of $X$. We call $\mathcal{T}_X$ the \emph{Toeplitz algebra} of $X$. There is a homomorphism $\psi^{(1)} : \mathcal{K}(X) \to B$ such that $\psi^{(1)}(\Theta_{x,y}) =\psi(x)\psi(y)^*$. Define a closed two-sided ideal of $A$ by $J_X:=\phi^{-1}(\mathcal{K}(X)) \cap (\ker\phi)^{\perp}$. The pair $(\psi,\pi)$ is \emph{covariant} if $\psi^{(1)}(\phi(a))=\pi(a)$ for all $a \in J_X$. \cite[Proposition~4.11]{Katsura:JFA04} shows that there is a $C^*$-algebra $\mathcal{O}_X$ generated by an injective universal covariant Toeplitz representation $(j_X,j_A)$ of $X$. We call $\mathcal{O}_X$ the \emph{Cuntz-Pimsner algebra} of $X$. Define $X^{\otimes 0}:=A$, define $X^{\otimes 1}:=X$, and inductively define $X^{\otimes n}:=X \otimes_A X^{\otimes n-1}$ for $n \geq 2$. The \emph{Fock space} $\mathcal{F}(X)$ of $X$ is the direct sum $\bigoplus_{n=0}^{\infty}X^{\otimes n}$. Define $\psi^{\otimes 0}:=\pi$ and $\psi^{\otimes 1}:=\psi$. For each $n \geq 2$, there is a linear map $\psi^{\otimes n}:X^{\otimes n} \to B$ such that $\psi^{\otimes n}(x \otimes \xi)=\psi(x)\psi^{\otimes n-1}(\xi)$ for all $x \in X$ and $\xi \in X^{\otimes n-1}$. Define $\psi^{(0)}:=\pi$. For each $n \geq 1$, since the pair $(\psi^{\otimes n},\pi)$ is a Toeplitz representation of $X^{\otimes n}$, there is a homomorphism $\psi^{(n)}:\mathcal{K}(X^{\otimes n}) \to B$ such that $\psi^{(n)}(\Theta_{\xi,\eta})=\psi^{\otimes n}(\xi)\psi^{\otimes n}(\eta)^*$. Finally, \cite[Proposition~2.7]{Katsura:JFA04} says that 
\begin{equation}\label{simplification of C^*(psi,pi)}
C^*(\psi,\pi)=\overline{\lsp}\{\psi^{\otimes n}(\xi)\psi^{\otimes m}(\eta)^*:\xi \in X^{\otimes n}, \eta \in X^{\otimes m} \}.
\end{equation}

In the rest of the section we recall the notion of a topological graph as studied by Katsura in \cite{Katsura:TAMS04}. A \emph{topological graph} is a quadruple $E=(E^0,E^1,r,s)$ such that $E^0, E^1$ are locally compact Hausdorff spaces, $r:E^1 \to E^0$ is a continuous map, and $s:E^1\to E^0$ is a local
homeomorphism. An \emph{$s$-section} is a subset $U \subset E^1$ such that $s \vert_U:U \to s(U)$ is a homeomorphism.

Given a topological graph $E$, Katsura in \cite{Katsura:TAMS04} defined a $C^*$-correspondence $X(E)$ over $C_0(E^0)$ called the \emph{graph correspondence associated to $E$}. For $x, y \in C_c(E^1),f \in C_0(E^0)$, and for $v \in E^0$, define $x\cdot f:=x (f \circ s), f \cdot x:=(f \circ r) x$, and $\langle x,y \rangle_{C_0(E^0)}(v):=\sum_{s(e)=v}\overline{x(e)}y(e)$. Then $C_c(E^1)$ is a right inner product $C_0(E^0)$-module, and the completion of $C_c(E^1)$ under the $\Vert\cdot\Vert_{C_0(E^0)}$-norm is the graph correspondence $X(E)$. The Toeplitz algebra of the graph correspondence $X(E)$ is denoted by $\mathcal{T}(E)$, and the Cuntz-Pimsner algebra of the graph correspondence $X(E)$ is denoted by $\mathcal{O}(E)$.

For a topological graph $E$, Katsura in \cite{Katsura:TAMS04} defined the following subsets of the vertex set $E^0$: the set $E_{\mathrm{sce}}^0:=E^0 \setminus \overline{r(E^1)}$ of sources; the set of finite receivers $E_{\mathrm{fin}}^0$ consisting of vertices $v$ with an open neighborhood $N$ of $v$ such that $r^{-1}(\overline{N})$ is compact; the set $E_{\mathrm{rg}}^0:=E_{\mathrm{fin}}^0 \setminus \overline{E_{\mathrm{sce}}^0}$ of regular vertices; and the set $E_{\mathrm{sg}}^0:=E^0 \setminus E_{\mathrm{rg}}^0$ of singular vertices.

\section{Twisted Topological Graph Algebras}

In this section, we firstly define the $1$-cocycles on a locally compact Hausdorff space from \cite{RaeburnWilliams:Moritaequivalenceand98}, and then in Theorem~\ref{C_0(E^0)-valued inner product on C_c(E,N,S)} we incorporate a $1$-cocycle into Katsura's construction of the graph correspondence associated to a topological graph to define the twisted graph correspondence.

Let $T$ be a locally compact Hausdorff space, let $\mathbf{N}=\{N_\alpha\}_{\alpha \in \Lambda}$ be an open cover of $T$. For $\alpha_1, \dots, \alpha_n \in \Lambda$, define
\[
N_{\alpha_1\dots\alpha_n} := \bigcap^n_{i=1} N_{\alpha_i}.
\]

\begin{defn}[{\cite[Definition~4.22]{RaeburnWilliams:Moritaequivalenceand98}}]\label{define a 1-cocycle relative to an open cover}
Let $T$ be a locally compact Hausdorff space, and let $\mathbf{N}=\{N_\alpha\}_{\alpha \in \Lambda}$ be an open cover of $T$. A collection of circle-valued continuous functions $\mathbf{S}=\{s_{\alpha\beta} \in C(\overline{N_{\alpha\beta}},\mathbb{T})\}_{\alpha, \beta \in \Lambda}$ is called a \emph{$1$-cocycle} relative to $\mathbf{N}$ if for $\alpha, \beta, \gamma \in \Lambda, s_{\alpha\beta}s_{\beta\gamma}=s_{\alpha\gamma}$ on $\overline{N_{\alpha\beta\gamma}}$. Suppose that $x, y \in \prod_{\alpha \in \Lambda} C(\overline{N_\alpha})$ satisfy $x_\alpha=s_{\alpha\beta}x_\beta$ and $y_\alpha=s_{\alpha\beta}y_\beta \ \mathrm{on}\ \overline{N_{\alpha\beta}}$ for all $\alpha, \beta \in \Lambda$. Define $[x \vert y] \in C(T)$ by
\begin{align*}
[x \vert y](t)=\overline{x_\alpha(t)}y_\alpha(t),\ \mathrm{if} \ t \in N_\alpha.
\end{align*} 
\end{defn}

\begin{defn}\label{define C_c(E,N,S)}
Let $E$ be a topological graph, let $\mathbf{N}=\{N_\alpha\}_{\alpha \in \Lambda}$ be an open cover of $E^1$, and let $\mathbf{S}=\{s_{\alpha\beta}\}_{\alpha,\beta \in \Lambda}$ be a $1$-cocycle relative to $\mathbf{N}$. Define
\begin{align*}
C_c({E,\mathbf{N},\mathbf{S}}):=\Big\{x \in \prod_{\alpha \in \Lambda} C(\overline{N_\alpha}): x_\alpha=s_{\alpha\beta}x_\beta \ \mathrm{on}\ \overline{N_{\alpha\beta}}, [x \vert x] \in C_c(E^1) \Big\}.
\end{align*}
For $x, y \in C_c({E,\mathbf{N},\mathbf{S}}), \alpha \in \Lambda, f \in C_0(E^0)$, and $v \in E^0$, define
\begin{enumerate}
\item\label{define x cdot f} $(x\cdot f)_\alpha:=x_\alpha (f \circ s \vert_{\overline{N_\alpha}})$;
\item\label{define langle x,y rangle_C_0(E^0)}$\langle x,y \rangle_{C_0(E^0)}(v):=\sum_{s(e)=v}[x\vert y](e)$; and
\item\label{define f cdot x}$(f \cdot x)_\alpha:=(f \circ r \vert_{\overline{N_\alpha}})x_\alpha$. 
\end{enumerate}
\end{defn}

We see that $C_c({E,\mathbf{N},\mathbf{S}})$ is a vector space under pointwise operations, and Properties~(\ref{define x cdot f}) and (\ref{define f cdot x}) of Definition~\ref{define C_c(E,N,S)} give right and left $C_0(E^0)$-actions on $C_c({E,\mathbf{N},\mathbf{S}})$, respectively. 

For $x, y \in C_c({E,\mathbf{N},\mathbf{S}})$, the polarization identity and \cite[Lemma~1.5]{Katsura:TAMS04} imply that
\begin{align*}
\langle x,y \rangle_{C_0(E^0)}&=\frac{1}{4}\sum_{n=0}^{3}(-i)^n\langle x+i^n y ,x+i^n y\rangle_{C_0(E^0)} 
\\&=\frac{1}{4}\sum_{n=0}^{3}(-i)^n \Big\langle \sqrt{[x+i^n y \vert x+i^n y]}, \sqrt{[x+i^n y \vert x+i^n y]}\Big\rangle_{C_0(E^0)}  \in C_c(E^0).
\end{align*}

It is easy to verify that $\langle\cdot,\cdot\rangle_{C_0(E^0)}$ in Definition~\ref{define C_c(E,N,S)} is a right $C_0(E^0)$-valued inner product on $C_c(E,\mathbf{N},\mathbf{S})$. We denote the completion of $C_c(E,\mathbf{N},\mathbf{S})$ under the $\Vert\cdot\Vert_{C_0(E^0)}$-norm by $X(E,\mathbf{N},\mathbf{S})$.

\begin{thm}\label{C_0(E^0)-valued inner product on C_c(E,N,S)}
Let $E$ be a topological graph, let $\mathbf{N}=\{N_\alpha\}_{\alpha \in \Lambda}$ be an open cover of $E^1$, and let $\mathbf{S}=\{s_{\alpha\beta}\}_{\alpha,\beta \in \Lambda}$ be a $1$-cocycle relative to $\mathbf{N}$. Then $X(E,\mathbf{N},\mathbf{S})$ is a $C^*$-correspondence over $C_0(E^0)$.
\end{thm}
\begin{proof}
For $x \in C_c(E,\mathbf{N},\mathbf{S})$, and $f \in C_0(E^0)$, we have
\begin{align*}
\Vert f \cdot x \Vert_{C_0(E^0)}^2&=\sup_{v \in E^0}\Big(\sum_{s(e)=v}\vert f(r(e)) \vert^2 [x \vert x](e)\Big)\leq\Vert f \Vert^2\sup_{v \in E^0}\Big(\sum_{s(e)=v} [x \vert x](e)\Big)
\\&= \Vert f \Vert^2 \Vert x \Vert_{C_0(E^0)}^2.
\end{align*}
So the left $C_0(E^0)$-action on $C_c(E,\mathbf{N},\mathbf{S})$ can be extended to $X(E,\mathbf{N},\mathbf{S})$. For $x,y \in C_c(E,\mathbf{N},\mathbf{S}), f \in C_0(E^0)$, and $v \in E^0$, we have
\begin{align*}
\langle f^* \cdot y,x \rangle_{C_0(E^0)}(v)&=\sum_{s(e)=v}[((f^* \circ r) \vert_{\overline{N_\alpha}}y_\alpha) \vert x](e)
=\sum_{s(e)=v}[y \vert ((f \circ r) \vert_{\overline{N_\alpha}} x_\alpha)](e)
\\&=\langle y,f \cdot x \rangle_{C_0(E^0)}(v).
\end{align*}
By continuity of the left action, $X(E,\mathbf{N},\mathbf{S})$ is a $C^*$-correspondence over $C_0(E^0)$. 
\end{proof}

\begin{defn}\label{def of T(E,N,S) and O(E,N,S)}
We call $X(E,\mathbf{N},\mathbf{S})$ the \emph{twisted graph correspondence} associated to the graph $E$ and the $1$-cocycle $\mathbf{S}$. We denote by $\mathcal{T}(E,\mathbf{N},\mathbf{S})$ the Toeplitz algebra of the twisted graph correspondence $X(E,\mathbf{N},\mathbf{S})$, and denote by $\mathcal{O}(E,\mathbf{N},\mathbf{S})$ the Cuntz-Pimsner algebra of $X(E,\mathbf{N},\mathbf{S})$.
\end{defn}

\begin{rmk}
Let $E$ be a topological graph, let $\mathbf{N}=\{N_\alpha\}_{\alpha \in \Lambda}$ be an open cover of $E^1$, and let $\mathbf{S}=\{s_{\alpha\beta}\}_{\alpha,\beta \in \Lambda}$ be a $1$-cocycle relative to $\mathbf{N}$. When $\mathbf{N}=\{E^1\}$ and $\mathbf{S}=\{1\}$, then $X(E,\mathbf{N},\mathbf{S})$ coincides with the ordinary graph correspondence $X(E)$. When $E^0=E^1$ and $r=s=\id$, Raeburn in \cite[Proposition~A3]{Raeburn:TAMS81} showed that $X(E,\mathbf{N},\mathbf{S})$ characterizes the $C_0(E^0)$--$C_0(E^0)$ imprimitivity bimodules with trivial Rieffel homeomorphism. When $E^0=E^1, r=\id$, and $s$ is a homeomorphism, in \cite[Theorem~3.1.16]{HLiPhDThesis} we proved a stronger result that $X(E,\mathbf{N},\mathbf{S})$ characterizes any $C_0(E^0)$--$C_0(E^0)$ imprimitivity bimodule.
\end{rmk}

\begin{eg}\label{Vasselli's eg}
Let $T$ be a compact Hausdorff space, let $\mathbf{N}=\{N_\alpha\}_{\alpha \in \Lambda}$ be an open cover of $T$, and let $\mathbf{S}=\{s_{\alpha\beta}\}_{\alpha,\beta \in \Lambda}$ be a $1$-cocycle relative to $\mathbf{N}$. Define a topological graph $E:=(T,T,\id,\id)$. The principal circle bundle $\mathbf{B}$ induced from the $1$-cocycle $\mathbf{S}$ is the quotient of $\amalg_{\alpha \in \Lambda}(N_\alpha \times \mathbb{T})$ by the equivalence relation $((t,z),\alpha) \sim ((t,z s_{\alpha\beta}(t)),\beta)$, which is a compact Hausdorff space (see \cite[Proposition~4.53, Example~4.58]{RaeburnWilliams:Moritaequivalenceand98}). Vasselli showed in \cite[Proposition~4.3]{MR1966825} that $\mathcal{O}(E,\mathbf{N},\mathbf{S}) \cong C(\mathbf{B})$.
\end{eg}

\begin{eg}
Example~\ref{Vasselli's eg} actually provides a class of examples that the twisted topological graph algebra of a topological graph may not be isomorphic to the untwisted topological graph algebra. For example, define a topological graph $E:=(S^2,S^2,\id,\id)$, the topological graph algebra $\mathcal{O}(E)$ is isomorphic to $C(S^2 \times \mathbb{T})$. Since $S^3$ is the Hopf circle bundle over $S^2$ with the projection $p:S^3 \to S^2$ by $p(z_1,z_2):=(2z_1z_2^*,\vert z_1 \vert^2-\vert z_2 \vert^2)$, there is a twisted topological graph algebra of $E$ which is isomorphic to $C(S^3)$. We can see that $S^3$ is not homeomorphic to $S^2 \times \mathbb{T}$ by calculating their fundamental groups. We have $\pi_1(S^3)=0$, and $\pi_1(S^2 \times \mathbb{T}) \cong \mathbb{Z}$.
\end{eg}

\begin{rmk}
The notion of $1$-cocycles in Definition~\ref{define a 1-cocycle relative to an open cover} comes from sheaf cohomology theory. Let $E$ be a topological graph. Each $1$-cocycle on $E^1$ canonically represents an element of the first cohomology group $H^1(E^1,\mathcal{S})$ (see \cite[Definition~4.22]{RaeburnWilliams:Moritaequivalenceand98}) and all $1$-cocycles on $E^1$ determine the group $H^1(E^1,\mathcal{S})$. In \cite[Theorem~3.3.3]{HLiPhDThesis} it was shown that two $1$-cocycles on $E^1$ representing the same element $H^1(E^1,\mathcal{S})$ give rise to isomorphic twisted graph correspondences. Therefore, from now on, we restrict our attention to covers of $E^1$ by precompact open $s$-sections.
\end{rmk}

Let $E$ be a topological graph, let $\mathbf{N}=\{N_\alpha\}_{\alpha \in \Lambda}$ be a cover of $E^1$ by precompact open $s$-sections, and let $\mathbf{S}=\{s_{\alpha\beta}\}_{\alpha,\beta \in \Lambda}$ be a $1$-cocycle relative to $\mathbf{N}$. The first look at the twisted graph correspondence $X(E,\mathbf{N},\mathbf{S})$ gives us the impression that its structure is maybe very complicated. As a matter of fact, every element in $C_c(E,\mathbf{N},\mathbf{S})$ is spanned by very simple elements. To establish this feature, we need to set up some notation. For $\alpha_0, \alpha \in \Lambda$, and for $f \in C_0(N_{\alpha_0})$, define $f^{\Ind_{\alpha_0}^{\alpha}} \in C(\overline{N_\alpha})$ by 
\begin{align*}
f^{\Ind_{\alpha_0}^{\alpha}}(t) := \begin{cases}
    s_{\alpha\alpha_0}(t)f(t) &\text{ if $t \in \overline{N_{\alpha\alpha_0}}$} \\
    0 &\text{ if $t \in \overline{N_\alpha} \setminus \overline{N_{\alpha\alpha_0}}$ }.
\end{cases}
\end{align*}
Then $(f^{\Ind_{\alpha_0}^{\alpha}})_{\alpha \in \Lambda} \in C_c(E,\mathbf{N},\mathbf{S})$. For $\alpha_1 \in \Lambda, g \in C_0(N_{\alpha_1})$, we have $[(f^{\Ind_{\alpha_0}^{\alpha}})\vert(g^{\Ind_{\alpha_1}^{\alpha}})]=(f^*g) \times s_{\alpha_0\alpha_1}\vert_{N_{\alpha_0\alpha_1}}$. Hence 
\[
\langle (f^{\Ind_{\alpha_0}^{\alpha}})_{\alpha \in \Lambda},(g^{\Ind_{\alpha_1}^{\alpha}})_{\alpha \in \Lambda}  \rangle_{C_0(E^0)}=\langle f,g \rangle_{C_0(E^0)} \times (s_{\alpha_0 \alpha_1}\vert_{N_{\alpha_0 \alpha_1}} \circ s\vert_{N_{\alpha_0 \alpha_1}}^{-1}).
\]

\begin{prop}\label{induced tuple generates C_c tuple}
Let $E$ be a topological graph, let $\mathbf{N}=\{N_\alpha\}_{\alpha \in \Lambda}$ be a cover of $E^1$ by precompact open $s$-sections, and let $\mathbf{S}=\{s_{\alpha\beta}\}_{\alpha,\beta \in \Lambda}$ be a $1$-cocycle relative to $\mathbf{N}$.  Then
\[
C_c(E,\mathbf{N},\mathbf{S})=\lsp\{(f^{\Ind_{\alpha_0}^{\alpha}})_{\alpha \in \Lambda}: \alpha_0 \in \Lambda, f \in C_c(N_{\alpha_0})\}.
\] 
\end{prop}
\begin{proof}
The inclusion $\supset$ is obvious since $\mathbf{N}$ consists of precompact open sets. For the reverse inclusion, fix $x \in C_c(E,\mathbf{N},\mathbf{S})$. For any $e \in \supp([x \vert x])$, there exists $\alpha_e \in \Lambda$, such that $e \in N_{\alpha_e}$. Since $\supp([x \vert x])$ is compact, there exists a finite subset $F \subset \supp([x \vert x])$ such that $\{N_{\alpha_e}\}_{e \in F}$ covers $\supp([x \vert x])$. We use a partition of unity (see \cite[Lemma~1.43]{Williams:Crossedproductsof07}) to get a finite collection of functions $\{h_e:e \in F \} \subset C_0(E^1)$ such that $\supp(h_e) \subset N_{\alpha_e}$ for all $e \in F$, and $\sum_{e \in F}h_e=1$ on $\supp([x \vert x])$. Then $x=\sum_{e \in F}(h_e \vert_{\overline{N_\alpha}} x_\alpha)_{\alpha \in \Lambda}$. Fix $e \in F$. Since $x_{\alpha_e} \in C(\overline{N_{\alpha_e}})$ and $\overline{N_{\alpha_e}}$ is compact, $x_{\alpha_e} \vert_{N_{\alpha_e}} \in C_b(N_{\alpha_e})$. So $h_e \times (x_{\alpha_e} \vert_{N_{\alpha_e}}) \in C_c(N_{\alpha_e})$, and $(h_e \vert_{\overline{N_\alpha}} x_\alpha)_{\alpha \in \Lambda}=((h_e \times (x_{\alpha_e} \vert_{N_{\alpha_e}}))^{\Ind_{\alpha_e}^{\alpha}})_{\alpha \in \Lambda}$.
\end{proof}

The following proposition generalizes \cite[Proposition~1.24]{Katsura:TAMS04}.

\begin{prop}\label{phi^{-1}(K(X)) intersects ker(phi^{perp}) of X_{E,N,S}}
Let $E$ be a topological graph, let $\mathbf{N}=\{N_\alpha\}_{\alpha \in \Lambda}$ be a cover of $E^1$ by precompact open $s$-sections, and let $\mathbf{S}=\{s_{\alpha}\}_{\alpha,\beta \in \Lambda}$ be a $1$-cocycle relative to $\mathbf{N}$. Then 
\[
J_{X(E,\mathbf{N},\mathbf{S})}=\phi^{-1}(\mathcal{K}(X(E,\mathbf{N},\mathbf{S}))) \cap (\ker\phi)^{\perp}=C_0(E_{\mathrm{rg}}^0).
\]
\end{prop}
\begin{proof}
For $f \in C_0(E^0)$, by the Urysohn's Lemma (see \cite[Lemma~1.41]{Williams:Crossedproductsof07}), $f \in \ker(\phi)$ if and only if $f \in C_0(E_{\mathrm{sce}}^0)$. So $(\ker\phi)^{\perp}=C_0(E^0 \setminus \overline{E_{\mathrm{sce}}^0})$.

We compute $\phi^{-1}(\mathcal{K}(X(E,\mathbf{N},\mathbf{S})))$. Fix a nonnegative $f \in C_c(E_{\mathrm{fin}}^0)$. By definition of $E_{\mathrm{fin}}^0$, the set $r^{-1}(\supp(f))$ is compact. For each $e \in r^{-1}(\supp(f))$, there exists $\alpha_e \in \Lambda$, such that $e \in N_{\alpha_e}$. Choose a finite subset $F \subset r^{-1}(\supp(f))$ such that $\{N_{\alpha_e}\}_{e \in F}$ covers $r^{-1}(\supp(f))$. We use a partition of unity to get a finite collection of functions $\{h_e:e \in F \}\subset C_0(E^1,[0,1])$ such that $\supp(h_e) \subset N_{\alpha_e}$ for all $e \in F$, and $\sum_{e \in F}h_e=1$ on $r^{-1}(\supp(f))$. A straightforward calculation gives 
\begin{equation}\label{computation of phi(f) in the twisted graph correspondence}
\phi(f)=\sum_{e \in F}\Theta_{ \big(\sqrt{h_ef \circ r}^{\mathrm{Ind}_{\alpha_e}^{\alpha}}\big ),\big( \sqrt{h_ef\circ r}^{\mathrm{Ind}_{\alpha_e}^{\alpha}} \big) }.
\end{equation}
So $C_0(E_{\mathrm{fin}}^0) \subset \phi^{-1}(\mathcal{K}(X(E,\mathbf{N},\mathbf{S})))$. Conversely, fix $f \in \phi^{-1}(\mathcal{K}(X(E,\mathbf{N},\mathbf{S})))$. Suppose for a contradiction that $f \notin C_0(E_{\mathrm{fin}}^0)$. Then there exists $v \notin E_{\mathrm{fin}}^0$ such that $f(v) \neq 0$. By continuity of $f$ there exists an open neighborhood $U$ of $v$ such that $\vert f(U)\vert \geq \epsilon>0$. For $x_1, \dots, x_n, y_1, \dots, y_n \in C_c(E,\mathbf{N},\mathbf{S})$, let $K=\bigcup^n_{i=1} \supp([x_i\vert x_i])\cup\supp([y_i\vert y_i])$. By definition of $E_{\mathrm{fin}}^0$, the set $r^{-1}(U)$ is not contained in $K$. So there exists $e \in N_{\alpha_0} \cap r^{-1}(U) \setminus K$. The Urysohn's Lemma yields $z \in C_c(E,\mathbf{N},\mathbf{S})$ satisfying $\Vert z \Vert_{C_0(E^0)}=1$ and $z_{\alpha_0}(e)=1$. So
\begin{align*}
\Big\Vert \phi(f) -\sum_{i=1}^{n}\Theta_{x_i,y_i} \Big\Vert^2 &\geq \Big[\phi(f)(z)-\sum_{i = 1}^{n}\Theta_{x_i,y_i}(z)  \Big\vert \phi(f)(z)-\sum_{i = 1}^{n}\Theta_{x_i,y_i}(z)  \Big](e)
\\&=\vert f(r(e))z_{\alpha_0}(e)\vert^2\geq\epsilon^2.
\end{align*}
It follows that $\phi(f) \notin \mathcal{K}(X(E,\mathbf{N},\mathbf{S}))$, which is a contradiction. So $\phi^{-1}(\mathcal{K}(X(E,\mathbf{N},\mathbf{S})))=C_0(E_{\mathrm{fin}}^0)$. By definition of $E_{\mathrm{rg}}^0$, we have $J_{X(E,\mathbf{N},\mathbf{S})}=C_0(E_{\mathrm{rg}}^0)$.
\end{proof}

\section{A Generalized Result of Vasselli}

Patani in \cite{MR2873418} conjectured that Vasselli's result in \cite[Proposition~4.3]{MR1966825} described in Example~\ref{Vasselli's eg} is still true when the compactness condition is lifted. The proof for the locally compact case does not appear to have been sorted out. So in this section we give a proof of this conjecture. Before we do that, we need a technical proposition.

\begin{prop}\label{existence of an inj uni cov twisted Toep rep}
Let $E$ be a topological graph, let $\mathbf{N}=\{N_\alpha\}_{\alpha \in \Lambda}$ be a cover of $E^1$ by precompact open $s$-sections, and let $\mathbf{S}=\{s_{\alpha\beta}\}_{\alpha,\beta \in \Lambda}$ be a $1$-cocycle relative to $\mathbf{N}$. Then there exist a collection of linear maps $\{\psi_\alpha: C_0(N_\alpha) \to \mathcal{O}(E,\mathbf{N},\mathbf{S}) \}_{\alpha \in \Lambda}$ and a homomorphism $\pi:C_0(E^0) \to \mathcal{O}(E,\mathbf{N},\mathbf{S})$ such that
\begin{enumerate}
\item\label{preserve the left right module action} $\psi_\alpha(x \cdot f)=\psi_\alpha(x)\pi(f), \psi_\alpha(f \cdot x)=\pi(f)\psi_\alpha(x)$ for $x \in C_0(N_\alpha), f \in C_0(E^0)$;
\item\label{preserve the inner product} $\psi_\alpha(x)^*\psi_\beta(y)=\pi(\langle x,y \rangle_{C_0(E^0)} \times (s_{\alpha\beta}\vert_{N_{\alpha\beta}} \circ s\vert_{N_{\alpha\beta}}^{-1}))$ for $x \in C_0(N_\alpha), y \in C_0(N_\beta)$;
\item\label{continuity of psi_alpha} each $\psi_\alpha$ is norm-decreasing under the supremum norm of $C_0(N_\alpha)$;
\item\label{transition property} $\psi_{\alpha}(x)=\psi_{\beta}(x \times (s_{\beta\alpha} \vert_{N_{\beta\alpha}}) )$ for $x \in C_0(N_{\alpha\beta})$;
\item\label{cov condition} for any nonnegative $f \in C_c(E^0_{\textrm{rg}})$, any finite subset $F \subset \Lambda$, and any collection $\{h_\alpha \in C_c(N_\alpha,[0,1])\}_{\alpha \in F}$ such that $\sum_{\alpha \in F}h_\alpha=1$ on $r^{-1}(\supp(f))$, we have 
\[
\pi(f) = \sum_{\alpha \in F} \psi_\alpha\big(\sqrt{h_\alpha f\circ r}\big)\psi_\alpha \big(\sqrt{h_\alpha f\circ r}\big)^*;
\]
\item\label{C^*(psi_alpha,pi)=O(E,N,S)} $C^*(\psi_\alpha,\pi)=\mathcal{O}(E,\mathbf{N},\mathbf{S})$; and
\item\label{universal property} if a collection of linear maps $\{\psi_\alpha': C_0(N_\alpha) \to B \}_{\alpha \in \Lambda}$ and a homomorphism $\pi':C_0(E^0) \to B$ satisfying Properties~(\ref{preserve the left right module action})--(\ref{cov condition}), then there exists a homomorphism $h:\mathcal{O}(E,\mathbf{N},\mathbf{S}) \to B$, such that $h \circ \psi_\alpha=\psi_\alpha'$ for all $\alpha \in \Lambda$, and $h \circ \pi=\pi'$.
\end{enumerate}
\end{prop}
\begin{proof}
Let $(\psi,\pi)$ be the injective universal covariant Toeplitz representation of $X(E,\mathbf{N},\mathbf{S})$ that generates $\mathcal{O}(E,\mathbf{N},\mathbf{S})$. For each $\alpha_0 \in \Lambda$, define a linear map $\psi_{\alpha_0}:C_0(N_{\alpha_0}) \to \mathcal{O}(E,\mathbf{N},\mathbf{S})$ by $\psi_{\alpha_0}(x):=\psi((x^{\Ind_{\alpha_0}^{\alpha}})_{\alpha \in \Lambda})$. It is straightforward to verify that $\{\psi_\alpha,\pi\}_{\alpha \in \Lambda}$ satisfies Properties~(\ref{preserve the left right module action})--(\ref{transition property}), (\ref{C^*(psi_alpha,pi)=O(E,N,S)}).

We check Property~(\ref{cov condition}). Fix a nonnegative function $f \in C_c(E^0_{\textrm{rg}})$, a finite subset $F \subset \Lambda$, and a collection $\{h_\alpha \in C_c(N_\alpha,[0,1])\}_{\alpha \in F}$ such that $\sum_{\alpha \in F}h_\alpha=1$ on $r^{-1}(\supp(f))$. Equation~(\ref{computation of phi(f) in the twisted graph correspondence}) gives
\[
\phi(f)=\sum_{\alpha \in F}\Theta_{ \big(\sqrt{h_\alpha f \circ r}^{\mathrm{Ind}_{\alpha}^{\beta}}\big )_{\beta \in \Lambda},\big( \sqrt{h_\alpha f\circ r}^{\mathrm{Ind}_{\alpha}^{\beta}} \big)_{\beta \in \Lambda} }.
\]
So by the covariance of $(\psi,\pi)$ and by definition of $\{\psi_\alpha\}_{\alpha \in \Lambda}$, we have
\begin{align*}
\pi(f)=\psi^{(1)}(\phi(f))&=\sum_{\alpha \in F}\psi\Big(\big(\sqrt{h_\alpha f \circ r}^{\mathrm{Ind}_{\alpha}^{\beta}}\big )_{\beta \in \Lambda}\Big)\psi\Big(\big(\sqrt{h_\alpha f \circ r}^{\mathrm{Ind}_{\alpha}^{\beta}}\big )_{\beta \in \Lambda}\Big)^*
\\&=\sum_{\alpha \in F} \psi_\alpha\big(\sqrt{h_\alpha f\circ r}\big)\psi_\alpha \big(\sqrt{h_\alpha f\circ r}\big)^*.
\end{align*}

Finally, we verify Property~(\ref{universal property}). Fix a collection of linear maps $\{\psi_\alpha': C_0(N_\alpha) \to B \}_{\alpha \in \Lambda}$ and a homomorphism $\pi':C_0(E^0) \to B$ satisfying Properties~(\ref{preserve the left right module action})--(\ref{cov condition}). For arbitrary $\alpha_1, \dots, \alpha_n \in \Lambda, x_i \in C_0(N_{\alpha_i})$, if $\sum_{i=1}^{n}(x_i^{\Ind_{\alpha_i}^{\alpha}})=0$, then 
\[
\Big\langle \sum_{i=1}^{n}(x_i^{\Ind_{\alpha_i}^{\alpha}}),\sum_{i=1}^{n}(x_i^{\Ind_{\alpha_i}^{\alpha}}) \Big\rangle_{C_0(E^0)}=\sum_{i,j=1}^{n}\langle x_i,x_j \rangle_{C_0(E^0)} \times (s_{\alpha_i \alpha_j}\vert_{N_{\alpha_i \alpha_j}} \circ s\vert_{N_{\alpha_i \alpha_j}}^{-1})=0.
\]
So
\[
\Big(\sum_{i=1}^{n}\psi_{\alpha_i}'(x_i)\Big)^*\Big(\sum_{i=1}^{n}\psi_{\alpha_i}'(x_i)\Big)=\sum_{i,j=1}^{n}\pi'(\langle x_i,x_j \rangle_{C_0(E^0)} \times (s_{\alpha_i \alpha_j}\vert_{N_{\alpha_i \alpha_j}} \circ s\vert_{N_{\alpha_i \alpha_j}}^{-1}))=0.
\]
Proposition~\ref{induced tuple generates C_c tuple} gives a Toeplitz representation $(\psi',\pi')$ of $X(E,\mathbf{N},\mathbf{S})$ in $B$ such that $\psi'((x^{\Ind_{\alpha_0}^{\alpha}})_{\alpha \in \Lambda})=\psi_{\alpha_0}'(x)$ for $\alpha_0 \in \Lambda$, and $x \in C_0(N_{\alpha_0})$. We show that $(\psi',\pi')$ is covariant. Fix a nonnegative function $f \in C_c(E^0_{\textrm{rg}})$. By definition of $E^0_{\textrm{rg}}$, we have $r^{-1}(\supp(f))$ is compact. For any $e \in r^{-1}(\supp(f))$, there exists $\alpha_e \in \Lambda$, such that $e \in N_{\alpha_e}$. There exists a finite subset $F \subset r^{-1}(\supp(f))$ such that $\{N_{\alpha_e}\}_{e \in F}$ covers $r^{-1}(\supp(f))$. We use a partition of unity to get a finite collection $\{h_e \in C_c(N_{\alpha_e},[0,1])\}_{e \in F}$ such that $\sum_{e \in F}h_e=1$ on $r^{-1}(\supp(f))$. Equation~(\ref{computation of phi(f) in the twisted graph correspondence}) gives
\[
\phi(f)=\sum_{e \in F}\Theta_{ \big(\sqrt{h_e f \circ r}^{\mathrm{Ind}_{\alpha_e}^{\alpha}}\big )_{\alpha \in \Lambda},\big(\sqrt{h_e f \circ r}^{\mathrm{Ind}_{\alpha_e}^{\alpha}}\big )_{\alpha \in \Lambda} }.
\]
By Property~(\ref{cov condition}) and by definition of $\psi'$, we have 
\begin{align*}
\pi'(f)&=\sum_{e \in F} \psi_{\alpha_e}'\big(\sqrt{h_e f\circ r}\big)\psi_{\alpha_e}' \big(\sqrt{h_e f\circ r}\big)^*=\sum_{e \in F}\psi'(\sqrt{h_e f \circ r}^{\mathrm{Ind}_{\alpha_e}^{\alpha}})\psi'(\sqrt{h_e f \circ r}^{\mathrm{Ind}_{\alpha_e}^{\alpha}})^*
\\&=\psi'^{(1)}(\phi(f)).
\end{align*}
So $(\psi',\pi')$ is covariant by Proposition~\ref{phi^{-1}(K(X)) intersects ker(phi^{perp}) of X_{E,N,S}}. By the universal property of $(\psi,\pi)$, there is a homomorphism $h:\mathcal{O}(E,\mathbf{N},\mathbf{S}) \to B$, such that $h \circ \psi=\psi'$ and $h \circ \pi=\pi'$. By definitions of $\{\psi_\alpha\}_{\alpha \in \Lambda}$ and $\psi'$, for $x \in C_0(N_{\alpha_0})$, we have
\[
h \circ \psi_{\alpha_0}(x)=h \circ \psi(x^{\mathrm{Ind}_{\alpha_0}^{\alpha}})=\psi'(x^{\mathrm{Ind}_{\alpha_0}^{\alpha}})=\psi_{\alpha_0}'(x). \qedhere
\]
\end{proof}

The following theorem is a generalization of a Vasselli's result in \cite[Proposition~4.3]{MR1966825}.

\begin{thm}\label{proof of Vasselli's result}
Let $T$ be a locally compact Hausdorff space, let $\mathbf{N}=\{N_\alpha\}_{\alpha \in \Lambda}$ be a cover of $T$ by precompact open sets, and let $\mathbf{S}=\{s_{\alpha\beta}\}_{\alpha,\beta \in \Lambda}$ be a $1$-cocycle relative to $\mathbf{N}$. Define a topological graph $E:=(T,T,\id,\id)$, and define a locally compact Hausdorff space $\mathbf{B}:=\amalg_{\alpha \in \Lambda}(N_\alpha \times \mathbb{T}) / ((t,z),\alpha) \sim ((t,z s_{\alpha\beta}(t)),\beta)$.  Then the twisted topological graph algebra $\mathcal{O}(E,\mathbf{N},\mathbf{S})$ is isomorphic to $C_0(\mathbf{B})$.
\end{thm}
\begin{proof}
By Proposition~\ref{existence of an inj uni cov twisted Toep rep}, there exist a collection of linear maps $\{\psi_\alpha: C_0(N_\alpha) \to \mathcal{O}(E,\mathbf{N},\mathbf{S}) \}_{\alpha \in \Lambda}$ and a homomorphism $\pi:C_0(E^0) \to \mathcal{O}(E,\mathbf{N},\mathbf{S})$ satisfying Properties~(\ref{preserve the left right module action})--(\ref{universal property}) of Proposition~\ref{existence of an inj uni cov twisted Toep rep}.

We prove that $\mathcal{O}(E,\mathbf{N},\mathbf{S})$ is commutative. Since the set $\{\psi_\alpha(x),\psi_\alpha(x)^*,\pi(f):\alpha \in \Lambda,x \in C_0(N_\alpha), f \in C_0(T)\}$ generates $\mathcal{O}(E,\mathbf{N},\mathbf{S})$, it is sufficient to show that these generators commute with each other. We check that each $\psi_\alpha(x)$ commutes with each $\psi_\beta(y)^*$; the othe commutation relations are straightforward. For $\alpha \in \Lambda$, and $x, y \in C_0(N_\alpha)$, we claim that $\psi_\alpha(x)\psi_\alpha(y)^*=\pi(xy^*)$. Suppose that $x ,y$ are both nonnegative and their supports are contained in $N_\alpha$. Simple calculation shows that $\psi_\alpha(x)\psi_\alpha(y)^*=\psi_\alpha(\sqrt{xy})\psi_\alpha(\sqrt{xy})^*$. Since $\{\psi_\alpha,\pi\}_{\alpha \in \Lambda}$ satisfies Property~(\ref{cov condition}) of Proposition~\ref{existence of an inj uni cov twisted Toep rep}, we have $\psi_\alpha(x)\psi_\alpha(y)^*=\pi(xy)$. The linearity and the continuity of $\psi_\alpha$ validate the claim. Take $\alpha, \beta \in \Lambda, x \in C_0(N_\alpha)$, and $y \in C_0(N_\beta)$. Let $(E_i)_{i \in I}$ be an approximate identity of $C_0(N_\alpha)$. Then $\psi_\alpha(x)\pi(E_i)\psi_\beta(y)^* \to \psi_\alpha(x)\psi_\beta(y)^*$ by continuity of $\psi_\alpha$. On the other hand, 
\begin{align*}
\psi_\alpha(x)\pi(E_i)\psi_\beta(y)^*&=\psi_\alpha(x)\psi_\beta(yE_i)^*
\\&=\psi_\alpha(x)\psi_\alpha((yE_i)\times(s_{\alpha\beta} \vert_{N_{\alpha\beta}}))^*
\\&=\pi(x ((y^*E_i)\times(s_{\beta\alpha} \vert_{N_{\alpha\beta}})) ) \text{ (by the claim) }
\\&=\psi_\beta(yE_i)^*\psi_\alpha(x)
\\&\to\psi_\beta(y)^*\psi_\alpha(x) \text{ (by continuity of $\psi_\beta$) }.
\end{align*}
So $\psi_\alpha(x)\psi_\beta(y)^*=\psi_\beta(y)^*\psi_\alpha(x)$. Hence $\mathcal{O}(E,\mathbf{N},\mathbf{S})$ is commutative.

We show the character space of $\mathcal{O}(E,\mathbf{N},\mathbf{S})$ is homeomorphic to $\mathbf{B}$. Fix a nonzero homomorphism $\varphi:\mathcal{O}(E,\mathbf{N},\mathbf{S}) \to \mathbb{C}$. We claim that $\varphi \circ \pi$ is not a zero homomorphism. Suppose it is a zero map, for a contradiction. Then for $x \in C_0(N_\alpha)$, we have 
\[
\varphi \circ \psi_\alpha(x)^*\varphi \circ \psi_\alpha(x)=\varphi(\pi(x^*x))=0.
\]
Since $C^*(\psi_\alpha,\pi)=\mathcal{O}(E,\mathbf{N},\mathbf{S})$, we have $\varphi=0$, which is a contradiction. So $\varphi \circ \pi$ is not a zero homomorphism. Then there exists $t \in T$ such that $\varphi \circ \pi(f)=f(t)$ for all $f \in C_0(T)$. Take $\alpha \in \Lambda$ such that $t \in N_\alpha$ and take $x \in C_0(N_\alpha)$ with $x(t)=1$, we have $\varphi \circ \psi_\alpha(x) \in \mathbb{T}$ since $\varphi$ is a homomorphism. If $t \in N_\beta$, and $y \in C_0(N_\beta)$ such that $y(t)=1$, then it is not hard to see that $\varphi(\psi_\alpha(x))s_{\alpha\beta}(t)=\varphi(\psi_\beta(y))$. So there is a well-defined map $\Gamma:\widehat{\mathcal{O}}(E,\mathbf{N},\mathbf{S}) \to \mathbf{B}$ such that $\Gamma(\varphi)=((t,\varphi\circ\psi_\alpha(x)),\alpha)$ for each $\varphi \in \widehat{\mathcal{O}}(E,\mathbf{N},\mathbf{S})$.

For $\varphi, \rho \in \widehat{\mathcal{O}}(E,\mathbf{N},\mathbf{S})$, if $\Gamma(\varphi)=\Gamma(\rho)$, then there exists $t \in T$, such that $\varphi \circ \pi(f)=\rho \circ \pi(f)=f(t)$, for all $f \in C_0(T)$. For any $\alpha \in \Lambda$, and for any $x \in C_0(N_\alpha)$, if $x(t) \neq 0$, then $t \in N_\alpha$. Since $\Gamma(\varphi)=\Gamma(\rho)$, by definition of $\Gamma$, we have $(t,\varphi \circ \psi_\alpha(x/x(t)),\alpha)=(t,\rho \circ \psi_\alpha(x/x(t)),\alpha)$. So $\varphi \circ \psi_\alpha(x)=\rho \circ \psi_\alpha(x)$. If $x(t)=0$, then $\varphi \circ \psi_\alpha(x)=\rho \circ \psi_\alpha(x)=0$. Hence $\varphi=\rho$ and $\Gamma$ is injective.

Take $t \in N_{\alpha_0}$, and $z \in \mathbb{T}$. For $\alpha \in \Lambda$, if $t \in N_\alpha$, then define a linear map $\psi_\alpha':C_0(N_\alpha) \to \mathbb{C}$ by $\psi_\alpha'(x):=zx(t)s_{\alpha_0\alpha}(t)$. If $t \notin N_\alpha$, then define $\psi_\alpha': C_0(N_\alpha) \to \mathbb{C}$ to be the zero map. Define a homomorphism $\pi':C_0(T) \to \mathbb{C}$ by $\pi'(f):=f(t)$. It is straightforward to see that $\{\psi_\alpha',\pi'\}_{\alpha \in \Lambda}$ satisfies Properties~(\ref{preserve the left right module action})--(\ref{transition property}) of Proposition~\ref{existence of an inj uni cov twisted Toep rep}. We prove Property~(\ref{cov condition}) of Proposition~\ref{existence of an inj uni cov twisted Toep rep} of $\{\psi_\alpha',\pi'\}_{\alpha \in \Lambda}$. For any nonnegative function $f \in C_c(T)$, any finite subset $F \subset \Lambda$, and any collection $\{h_\alpha \in C_c(N_\alpha,[0,1])\}_{\alpha \in F}$ such that $\sum_{\alpha \in F}h_\alpha=1$ on $\supp(f)$, we have 
\begin{align*}
\sum_{\alpha \in F}\psi_\alpha'(\sqrt{h_\alpha f})\psi_\alpha'(\sqrt{h_\alpha f})^*=\sum_{\alpha \in F}h_\alpha(t)f(t)=f(t)=\pi'(t).
\end{align*}
Since $\{\psi_\alpha,\pi\}_{\alpha \in \Lambda}$ satisfies Property~(\ref{universal property}) of Proposition~\ref{existence of an inj uni cov twisted Toep rep}, there exists a homomorphism $\varphi:\mathcal{O}(E,\mathbf{N},\mathbf{S}) \to \mathbb{C}$, such that $\varphi \circ \psi_\alpha=\psi_\alpha'$ for all $\alpha \in \Lambda$, and $\varphi \circ \pi=\pi'$. So $\varphi$ is a nonzero homomorphism. The Urysohn's Lemma gives $x \in C_0(N_{\alpha_0})$ such that $x(t)=1$. Then $\varphi \circ \psi_{\alpha_0}(x)=\psi_{\alpha_0}'(x)=zx(t)s_{\alpha_0\alpha_0}(t)=z$. So $\Gamma(\varphi)=((t,z),\alpha_0)$, and $\Gamma$ is a bijection with the inverse $\Gamma^{-1}((t,z),\alpha_0)=\varphi$.

Next we prove that $\Gamma$ is continuous. Fix a convergent net $(\varphi_a)_{a \in A} \subset \widehat{\mathcal{O}}(E,\mathbf{N},\mathbf{S})$ with the limit $\varphi$. Then there exist $t_a, t \in T$ such that $\varphi_a \circ \pi(f)=f(t_a)$ and $\varphi \circ \pi(f)=f(t)$ for all $f \in C(T)$. Since $\varphi_a \to \varphi$, we have $\varphi_a \circ \pi \to \varphi \circ \pi$. So $t_a \to t$. It is fine to suppose that there exist $\alpha_0 \in \Lambda$ and a compact set $K \subset N_{\alpha_0}$ such that $t_a,t \in K$. By the Urysohn's Lemma, there is $x \in C_0(N_{\alpha_0})$ such that $x(K)=1$. By Definition of $\Gamma$, we have $\Gamma(\varphi_a)=((t_a,\varphi_a\circ\psi_{\alpha_0}(x)),\alpha_0)$ and $\Gamma(\varphi)=((t,\varphi \circ \psi_{\alpha_0}(x)),\alpha_0)$. Hence $\Gamma(\varphi_a) \to \Gamma(\varphi)$ because $\varphi_a \to \varphi$.

Finally we show that $\Gamma$ is open. Fix a convergent net $((t_a,z_a),\alpha_a)_{a \in A} \to ((t,z),\alpha_0)$ in $\mathbf{B}$. Let $\Gamma^{-1}((t_a,z_a),\alpha_a)=\varphi_a$, and let $\Gamma^{-1}((t,z),\alpha_0)=\varphi$. Fix $\alpha \in \Lambda$, and fix $x \in C_0(N_\alpha)$. Suppose that $t \in N_\alpha$. Then there exists $a_0 \in A$, such that $t_a \in N_{\alpha_0\alpha}$ whenever $a \geq a_0, (t_a)_{a \geq a_0} \to t$, and $(s_{\alpha_a \alpha_0}(t_a)z_a)_{a \geq a_0} \to z$. By definition of $\Gamma^{-1}$, we have when $a \geq a_0$,
\begin{align*}
\varphi_a \circ \psi_\alpha(x)=z_ax(t_a)s_{\alpha_a \alpha}(t_a)=z_ax(t_a)s_{\alpha_a\alpha_0}(t_a)s_{\alpha_0\alpha}(t_a) \to zx(t)s_{\alpha_0\alpha}(t)=\varphi \circ \psi_\alpha(x).
\end{align*}
Suppose that $t \notin N_\alpha$. By definition of $\Gamma^{-1}$, we have $\varphi \circ \psi_\alpha(x)=0$, and $\vert \varphi_a \circ \psi_\alpha(x)\vert=\vert x(t_a) \vert \to \vert x(t)\vert=0$. Fix $f \in C_0(T)$. Then $\varphi_a \circ \pi(f)=f(t_a) \to f(t)= \varphi \circ \pi(f)$. Hence $\varphi_a \to \varphi$ because $C^*(\psi_\alpha,\pi)$ generates $\mathcal{O}(E,\mathbf{N},\mathbf{S})$. 
\end{proof}

\section{Technical Results}

In this section, we develop technical tools that we will need in later sections. These are analogous to technical results of Katsura for ordinary topological graph algebras \cite{Katsura:TAMS04, Katsura:IJM06, Katsura:ETDS06}.

Throughout this section, we fix a topological graph $E$, a cover $\mathbf{N}=\{N_\alpha\}_{\alpha \in \Lambda}$ of $E^1$ by precompact open $s$-sections, and a $1$-cocycle $\mathbf{S}=\{s_{\alpha\beta}\}_{\alpha,\beta \in \Lambda}$ relative to $\mathbf{N}$. 

First of all, we connect the Fock space of the twisted graph correspondence $X(E,\mathbf{N},\mathbf{S})$ with the finite-path space of $E$. 

Let $n \geq 1$. Define the finite-path space 
\[
E^n:=\{(e_1,\dots,e_n) \in \prod_{i=1}^{n}E^1: s(e_i)=r(e_{i+1}), i=1,\dots,n-1 \}.
\]
Define $r^n:E^n \to E^0$ by $r^n(e_1,\dots,e_n):=r(e_1)$, define $s^n:E^n \to E^0$ by $s^n(e_1,\dots,e_n):=s(e_n)$. Then $E_n:=(E^0,E^n,r^n,s^n)$ is a topological graph. Define a cover $\mathbf{N}^n:=\{(N_{\alpha_1} \times \cdots \times N_{\alpha_n}) \cap E^n\}_{\alpha_1,\dots,\alpha_n \in \Lambda}$ of $E^n$ by precompact open $s^n$-sections. Define a $1$-cocycle $\mathbf{S}^n:=\{s_{\alpha_1 \beta_1} \diamond \cdots \diamond s_{\alpha_n \beta_n}\}$ relative to $\mathbf{N}^n$ by $s_{\alpha_1 \beta_1} \diamond \cdots \diamond s_{\alpha_n \beta_n}((e_i)_{i=1}^{n}):=s_{\alpha_1 \beta_1}(e_1) \cdots s_{\alpha_n \beta_n}(e_n)$ for all $(e_i)_{i=1}^{n} \in \overline{N_{\alpha_1 \beta_1}\times \cdots \times N_{\alpha_n \beta_n} \cap E^n}$. For $x_1, \dots, x_n \in C_c(E_n,\mathbf{N}^n,\mathbf{S}^n), \alpha_1, \dots, \alpha_n \in \Lambda$, and $(e_1,\dots,e_n) \in \overline{N_{\alpha_1}\times\dots\times N_{\alpha_n} \cap E^n}$, define $(x_1 \diamond\dots\diamond x_n)_{\alpha_1,\dots,\alpha_n}(e_1,\dots,e_n):=x_{1,\alpha_1}(e_1)\dots x_{n,\alpha_n}(e_n)$. Then $x_1 \diamond\dots\diamond x_n \in C_c(E_n,\mathbf{N}^n,\mathbf{S}^n)$. Let $C_c(E_0,\mathbf{N}^0,\mathbf{S}^0):=C_c(E^0)$, and $X(E_0,\mathbf{N}^0,\mathbf{S}^0):=C_0(E^0)$.

The following proposition is a generalization of \cite[Proposition~1.27]{Katsura:TAMS04}.

\begin{prop}\label{X(E,N,S)^{otimes n} isomorphic with X(E_n,N^n,S^n)}
For each $n \geq 1$, there exists an isomorphism $\Phi_n: X(E,\mathbf{N},\mathbf{S})^{\otimes n} \to X(E_n,\mathbf{N}^n,\mathbf{S}^n)$ of $C^*$-correspondences over $C_0(E^0)$ such that $\Phi_n(x_1 \otimes\dots\otimes x_n)=x_1 \diamond\dots\diamond x_n$, for all $x_1, \dots, x_n \in C_c(E,\mathbf{N},\mathbf{S})$. Moreover, 
\begin{equation}\label{compute C_c(E_n,N^n,S^n)}
C_c(E_n,\mathbf{N}^n,\mathbf{S}^n)=\lsp\{x_1 \diamond\dots\diamond x_n: x_1,\dots,x_n \in C_c(E,\mathbf{N},\mathbf{S}) \}.
\end{equation}
Hence
\begin{equation}\label{compute X(E_n,N^n,S^n)}
X(E_n,\mathbf{N}^n,\mathbf{S}^n)=\overline{\lsp}\{x_1 \diamond\dots\diamond x_n: x_1,\dots,x_n \in C_c(E,\mathbf{N},\mathbf{S}) \}.
\end{equation}
\end{prop}
\begin{proof}
We prove this theorem by the induction argument on $n \geq 1$. When $n=1$, the result is obvious. Suppose that the theorem holds for $n \geq 1$, we prove that the theorem is true for $n+1$. We aim to define a map $\varphi:C_c(E,\mathbf{N},\mathbf{S}) \odot_{C_0(E^0)} C_c(E_n,\mathbf{N}^n,\mathbf{S}^n) \to X(E_{n+1},\mathbf{N}^{n+1},\mathbf{S}^{n+1})$ by $\varphi(x \odot y):=x \diamond y$ for all $x \in C_c(E,\mathbf{N},\mathbf{S}), y \in C_c(E_n,\mathbf{N}^n,\mathbf{S}^n)$. Straightforward calculation shows that $\varphi$ is well-defined and preserves the right inner products. So we can extend $\varphi$ uniquely to $X(E,\mathbf{N},\mathbf{S}) \otimes_{C_0(E^0)} X(E,\mathbf{N},\mathbf{S})^{\otimes n}$ and the unique extension $\varphi$ preserves the right inner products. Fix $f \in C_c((N_{\alpha_0} \times N_{\alpha_1} \times\dots\times N_{\alpha_n}) \cap E^{n+1})$. Let $U:=(N_{\alpha_0} \times N_{\alpha_1} \times\dots\times N_{\alpha_n}) \cap E^{n+1}$, and let $V:=(N_{\alpha_1} \times\dots\times N_{\alpha_n}) \cap E^n$. Denote the two projections by $P:U \to N_{\alpha_0}$ and $Q:U \to V$. The Urysohn's Lemma gives $g \in C_0(N_{\alpha_0})$ such that $g(P(\supp(f)))=1$. Since $\mathbf{N}$ consists of $s$-sections, there exists $h \in C_0(V)$ such that $h(e_1,\dots,e_n)=f(s \vert_{N_{\alpha_0}}^{-1}(r^n(e_1,\dots,e_n)),e_1,\dots,e_n)$ for all $e \in (r^{n})^{-1}(s(N_{\alpha_0})) \cap V$. So
\[
\varphi((g^{\mathrm{Ind}_{\alpha_0}^{\beta_0}})\otimes (h^{\mathrm{Ind}_{\alpha_1,\dots,\alpha_n}^{\beta_1,\dots,\beta_n}}))=(g^{\mathrm{Ind}_{\alpha_0}^{\beta_0}}) \diamond (h^{\mathrm{Ind}_{\alpha_1,\dots,\alpha_n}^{\beta_1,\dots,\beta_n}})=(f^{\mathrm{Ind}_{\alpha_0,\alpha_1,\dots,\alpha_n}^{\beta_0,\beta_1,\dots,\beta_n}}).
\]
Proposition~\ref{induced tuple generates C_c tuple} and the induction assumption imply the case for $n+1$.
\end{proof}

Let $(\psi,\pi)$ be a Toeplitz representation of $X(E,\mathbf{N},\mathbf{S})$ in a $C^*$-algebra $B$. The following results are immediate consequences of Proposition~\ref{X(E,N,S)^{otimes n} isomorphic with X(E_n,N^n,S^n)}. Define $\psi_0:=\pi$, and define $\psi_0^{(1)}:=\pi$. For each $n \geq 1$, the pair $(\psi_n,\pi)$ is a Toeplitz representation of $X(E_n,\mathbf{N}^n,\mathbf{S}^n)$ such that $\psi_n(x_1 \diamond\dots\diamond x_n):=\psi(x_1)\dots\psi(x_n)$ for all $x_1, \dots, x_n \in C_c(E,\mathbf{N},\mathbf{S})$. Let $\psi_n^{(1)}:\mathcal{K}(X(E_n,\mathbf{N}^n,\mathbf{S}^n)) \to B$ be the homomorphism such that $\psi_n^{(1)}(\Theta_{\xi,\eta}):=\psi_n(\xi)\psi_n(\eta)^*$ for all $\xi, \eta \in X(E_n,\mathbf{N}^n,\mathbf{S}^n)$. Then $\psi_n^{(1)}$ is injective whenever $\pi$ is injective. For each $n \geq 0$, define $B_n$ to be the image of $\psi_n^{(1)}$, define $B_{[0,n]}:=B_0+\dots+B_n$. Define $B_{[0,\infty]}:=\overline{\bigcup_{n=0}^{\infty}B_{[0,n]}}$, which is called the \emph{core} of $C^*(\psi,\pi)$ (The $C^*$-subalgebras $B_n, B_{[0,n]}, B_{[0,\infty]}$ coincide with Katsura's definitions in \cite{Katsura:JFA04}). We have
\begin{align*}
C^*(\psi,\pi)&=\overline{\lsp}\{\psi_n(\xi)\psi_m(\eta)^*: \xi \in C_c(E_n,\mathbf{N}^n,\mathbf{S}^n), \eta \in C_c(E_m,\mathbf{N}^m,\mathbf{S}^m)\}
\\&=\overline{\lsp}\{\psi_n(\xi)\psi_m(\eta)^*:
\\& \text{ if } n \geq 1, \text{ then } \xi=x_1 \diamond\dots\diamond x_n, \text{ where } x_1,\dots,x_n \in C_c(E,\mathbf{N},\mathbf{S});
\\& \text{ if } m \geq 1, \text{ then } \eta=y_1 \diamond\dots\diamond y_m, \text{ where } y_1,\dots,y_m \in C_c(E,\mathbf{N},\mathbf{S});
\\& \text{ if } n=0, \text{ then } \xi \in C_c(E^0); \text{ and if } m=0, \text{ then } \eta \in C_c(E^0)
\}.
\end{align*}

Secondly, we prove a version of the Tietze extension theorem for the twisted graph correspondence $X(E,\mathbf{N},\mathbf{S})$. Then we use this result to construct a very useful homomorphism in Proposition~\ref{omega is a homomorphism}.

Let $F^0$ be a closed set of $E^0$, and let $F^1:=s^{-1}(F^0)$. The restriction $s\vert_{F^1}:F^1 \to F^0$ is a local homeomorphism. Define a precompact open cover $\mathbf{N}^{F^1}:=\{N_\alpha \cap F^1 \}_{\alpha \in \Lambda}$ of $F^1$, and define a $1$-cocycle $\mathbf{S}^{F^1}:=\{s_{\alpha\beta}^{F^1}:=s_{\alpha\beta} \vert_{\overline{N_{\alpha\beta}\cap F^1}}\}_{\alpha,\beta \in \Lambda}$ relative to $\mathbf{N}^{F^1}$. Let
\begin{align*}
C_c(F^1,\mathbf{N}^{F^1},\mathbf{S}^{F^1}):=\Big\{x \in \prod_{\alpha \in \Lambda} C(\overline{N_\alpha \cap F^1}): x_\alpha=s_{\alpha\beta}^{F^1}x_\beta \ \mathrm{on} \ \overline{N_{\alpha\beta} \cap F^1},
[x \vert x] \in C_c(F^1) \Big\}.
\end{align*}
Conditions~(\ref{define x cdot f}), (\ref{define langle x,y rangle_C_0(E^0)}) of Definition~\ref{define C_c(E,N,S)} are the right action and the right $C_0(F^0)$-valued inner product on $C_c(F^1,\mathbf{N}^{F^1},\mathbf{S}^{F^1})$ by Theorem~\ref{C_0(E^0)-valued inner product on C_c(E,N,S)}. Denote its completion under the $\Vert\cdot\Vert_{C_0(F^0)}$-norm by $X(F^1,\mathbf{N}^{F^1},\mathbf{S}^{F^1})$. 

\begin{prop}\label{Katsura's Ext Thm}
Fix a closed subset $F^0 \subset E^0$ and let $F^1=s^{-1}(F^0)$. For any $x \in C_c(F^1,\mathbf{N}^{F^1},\mathbf{S}^{F^1})$, there exists $y \in C_c(E,\mathbf{N},\mathbf{S})$, such that $y_\alpha \vert_{\overline{N_\alpha \cap F^1}}=x_\alpha$ for all $\alpha \in \Lambda$, and $\Vert y \Vert_{C_0(E^0)}=\Vert x \Vert_{C_0(F^0)}$.
\end{prop}
\begin{proof}
Fix $\alpha_0 \in \Lambda$, and fix $f \in C(F^1)$ with $\supp(f) \subset N_{\alpha_0} \cap F^1$. By the Tietze extension theorem (see \cite[Lemma~1.42]{Williams:Crossedproductsof07}), there exists $g \in C(E^1)$ such that $f=g$ on $\overline{N_{\alpha_0} \cap F^1}$. By the Urysohn's lemma, there exists $h \in C_0(N_{\alpha_0})$ such that $h(\supp(f))=1$. Then $gh=f$ on $\overline{N_{\alpha_0} \cap F^1}$ and $gh \in C_0(N_{\alpha_0})$. So $(gh)^{\Ind_{\alpha_0}^{\alpha}} \vert_{\overline{N_{\alpha_0}\cap F^1}}=f^{\Ind_{\alpha_0}^{\alpha}}$ for all $\alpha \in \Lambda$. Hence Proposition~\ref{induced tuple generates C_c tuple} implies that for any $x \in C_c(F^1,\mathbf{N}^{F^1},\mathbf{S}^{F^1})$, there exists $y \in C_c(E,\mathbf{N},\mathbf{S})$, such that $y_\alpha \vert_{\overline{N_\alpha \cap F^1}}=x_\alpha$, for all $\alpha \in \Lambda$.

Thus we can extend each element in $C_c(F^1,\mathbf{N}^{F^1},\mathbf{S}^{F^1})$ to $C_c(E,\mathbf{N},\mathbf{S})$. Next we need to show the existence of the extension preserving the norms. Katsura's proof of \cite[Lemma~1.11]{Katsura:TAMS04} fits in our case.

Fix $x \in C_c(F^1,\mathbf{N}^{F^1},\mathbf{S}^{F^1})$. If $x=0$ then $y=0$ does the job, so we suppose that $x \neq 0$. Take $y \in C_c(E,\mathbf{N},\mathbf{S})$, such that $y_\alpha \vert_{\overline{N_\alpha \cap F^1}}=x_\alpha$, for all $\alpha \in \Lambda$. For each $\alpha \in \Lambda$, define $z_\alpha:\overline{N_\alpha} \to \mathbb{C}$ by 
\[
z_\alpha(e):=y_\alpha(e)\Vert x \Vert_{C_0(F^0)}/(\max\{\Vert x \Vert_{C_0(F^0)}^2, \langle y, y \rangle_{C_0(E^0)} (s(e))\})^{1/2}. 
\]
Then $z:=(z_\alpha)_{\alpha \in \Lambda}$ does the job.
\end{proof}

Let $F^0$ be a closed subset of $E^0$ and let $F^1=s^{-1}(F^0)$. Let $T \in \mathcal{L}(X(E,\mathbf{N},\mathbf{S}))$, and let $x \in C_c(F^1,\mathbf{N}^{F^1},\mathbf{S}^{F^1})$. By Proposition~\ref{Katsura's Ext Thm}, take $y, z \in C_c(E,\mathbf{N},\mathbf{S})$ such that $y_\alpha \vert_{\overline{N_\alpha \cap F^1}}=z_\alpha \vert_{\overline{N_\alpha \cap F^1}}=x_\alpha$, and take $(y_n),(z_n) \subset C_c(E,\mathbf{N},\mathbf{S})$ such that $y_n \to Ty, z_n \to Tz$. Then $y-z \in X(E,\mathbf{N},\mathbf{S})_{C_0(E^0 \setminus F^0)}$, which implies that $T(y-z) \in X(E,\mathbf{N},\mathbf{S})_{C_0(E^0 \setminus F^0)}$. So 
\[
\langle y_n-z_n,y_n-z_n \rangle_{C_0(E^0)} \to \langle T(y-z),T(y-z) \rangle_{C_0(E^0)} \in C_0(E^0 \setminus F^0).
\]
The sequences $(y_{n,\alpha}\vert_{\overline{N_\alpha \cap F^1}})_{\alpha \in \Lambda}, (z_{n,\alpha}\vert_{\overline{N_\alpha \cap F^1}})_{\alpha \in \Lambda}$ converge in $X(F^1,\mathbf{N}^{F^1},\mathbf{S}^{F^1})$. Since 
\[
\langle ((y_{n,\alpha}-z_{n,\alpha}) \vert_{\overline{N_\alpha \cap F^1}})_{\alpha \in \Lambda},((y_{n,\alpha}-z_{n,\alpha}) \vert_{\overline{N_\alpha \cap F^1}})_{\alpha \in \Lambda} \rangle_{C_0(F^0)}(v) \to 0 \text{ for all } v \in F^0,
\]
we have $((y_{n,\alpha}-z_{n,\alpha}) \vert_{\overline{N_\alpha \cap F^1}})_{\alpha \in \Lambda} \to 0$. Define $\omega(T):C_c(F^1,\mathbf{N}^{F^1},\mathbf{S}^{F^1}) \to X(F^1,\mathbf{N}^{F^1},\mathbf{S}^{F^1})$ by $\omega(T)(x):=\lim_{n \to \infty}(y_{n,\alpha}\vert_{\overline{N_\alpha \cap F^1}})_{\alpha \in \Lambda}$. It is straightforward to show that $\omega(T)$ is bounded and linear, and the unique extension $\omega(T)$ is adjointable with $\omega(T)^*=\omega(T^*)$.

The following proposition provides a generalization, and a detailed proof of an assertion made on \cite[Page~4294]{Katsura:TAMS04}.

\begin{prop}\label{omega is a homomorphism}
Fix a closed subset $F^0 \subset E^0$ and let $F^1=s^{-1}(F^0)$. The map $\omega: \mathcal{L}(X(E,\mathbf{N},\mathbf{S})) \to \mathcal{L}(X(F^1,\mathbf{N}^{F^1},\mathbf{S}^{F^1}))$ is a homomorphism, and
\begin{equation}\label{compute of the kernel of omega}
\ker(\omega)=\{T : Tx \in X(E,\mathbf{N},\mathbf{S})_{C_0(E^0 \setminus F^0)} \text{ for all } x \in X(E,\mathbf{N},\mathbf{S})\}.
\end{equation}
\end{prop}
\begin{proof}
We show that $\omega$ is a homomorphism. The linearity of $\omega$ is straightforward, and $\omega$ preserves adjoints since we just showed that $\omega(T)^*=\omega(T^*)$. Let us prove that $\omega$ preserves multiplication.

Fix $T, S \in \mathcal{L}(X(E,\mathbf{N},\mathbf{S}))$ and $x \in C_c(F^1,\mathbf{N}^{F^1},\mathbf{S}^{F^1})$. Take $y \in C_c(E,\mathbf{N},\mathbf{S})$ with $y_\alpha \vert_{\overline{N_\alpha \cap F^1}}=x_\alpha$ for all $\alpha \in \Lambda$, and choose $(y_n) \subset C_c(E,\mathbf{N},\mathbf{S})$ with $y_n \to Sy$. Then $Ty_n \to TSy, (y_{n,\alpha} \vert_{\overline{N_\alpha \cap F^1}})_{\alpha \in \Lambda} \to \omega(S)(x)$, and $\omega(T)(y_{n,\alpha} \vert_{\overline{N_\alpha \cap F^1}})_{\alpha \in \Lambda} \to \omega(T)\omega(S)(x)$. For each $n \geq 1$, there exists $z_n \in C_c(E,\mathbf{N},\mathbf{S})$, such that $\Vert z_n-Ty_n \Vert_{C_0(E^0)} < 1/n$. So $z_n \to TSy$, and $(z_{n,\alpha} \vert_{\overline{N_\alpha \cap F^1}}) \to \omega(T)\omega(S)(x)$. Hence $(z_{n,\alpha} \vert_{\overline{N_\alpha \cap F^1}}) \to \omega(TS)(x)$, and $\omega(TS)(x)=\omega(T)\omega(S)(x)$.

Now we compute the kernel of $\omega$. Fix $T \in \ker(\omega)$ and $x \in C_c(E,\mathbf{N},\mathbf{S})$. Take $(x_n) \subset C_c(E,\mathbf{N},\mathbf{S})$ such that $x_n \to Tx$. So
\[
\omega(T)((x_\alpha \vert_{\overline{N_\alpha \cap F^1}})_{\alpha \in \Lambda})=\lim_{n \to \infty}(x_{n,\alpha} \vert_{\overline{N_\alpha \cap F^1}})_{\alpha \in \Lambda}=0.
\]
For any $v \in F^0$, we have 
\[
\langle x_n,x_n \rangle_{C_0(E^0)}(v) =\langle (x_{n,\alpha} \vert_{\overline{N_\alpha \cap F^1}}),(x_{n,\alpha} \vert_{\overline{N_\alpha \cap F^1}}) \rangle_{C_0(F^0)}(v) \to 0.
\]
Hence $\langle Tx,Tx \rangle_{C_0(E^0)} \in C_0(E^0 \setminus F^0)$ and $Tx \in X(E,\mathbf{N},\mathbf{S})_{C_0(E^0 \setminus F^0)}$. By continuity of $T, Tx \in X(E,\mathbf{N},\mathbf{S})_{C_0(E^0 \setminus F^0)}$ for all $x \in X(E,\mathbf{N},\mathbf{S})$. Conversely, fix $T$ such that $Tx \in X(E,\mathbf{N},\mathbf{S})_{C_0(E^0 \setminus F^0)}$ for all $x \in X(E,\mathbf{N},\mathbf{S})$, and fix $x \in C_c(F^1,\mathbf{N}^{F^1},\mathbf{S}^{F^1})$. Take $y \in C_c(E,\mathbf{N},\mathbf{S})$ with $y_\alpha \vert_{\overline{N_\alpha \cap F^1}}=x_\alpha$ for all $\alpha \in \Lambda$, and choose $(y_n) \subset C_c(E,\mathbf{N},\mathbf{S})$ with $y_n \to Ty$. Then 
\begin{align*}
\langle \omega(T)(x),\omega(T)(x) \rangle_{C_0(F^0)}=\lim_{n \to \infty}\langle y_n,y_n \rangle_{C_0(E^0)} \vert_{F^0}=\langle Ty,Ty \rangle_{C_0(E^0)} \vert_{F^0}=0.
\end{align*}
So $T \in \ker(\omega)$.
\end{proof}

The following proposition is a generalization of \cite[Lemma~1.14]{Katsura:TAMS04}.

\begin{prop}\label{omega maps K(X(E,N,S)) onto K(X(X^1,N^X^1,S^X^1))}
Fix a closed subset $F^0 \subset E^0$ and let $F^1=s^{-1}(F^0)$. Then
\begin{enumerate}
\item\label{Theta_{x,y}=Theta_{x_{X^1},y_{X^1}}} $\omega(\Theta_{x,y})=\Theta_{(x_\alpha \vert_{\overline{N_\alpha \cap F^1}}),(y_\alpha \vert_{\overline{N_\alpha \cap F^1}})}$, for all $x, y \in C_c(E,\mathbf{N},\mathbf{S})$;
\item\label{image of omega on K(X(E,N,S))} $\omega(\mathcal{K}(X(E,\mathbf{N},\mathbf{S})))=\mathcal{K}(X(F^1,\mathbf{N}^{F^1},\mathbf{S}^{F^1}))$; and
\item\label{ker of omega on K(X(E,N,S))} $\ker(\omega) \cap \mathcal{K}(X(E,\mathbf{N},\mathbf{S}))=\overline{\lsp}\{\Theta_{x,y}: x, y \in X(E,\mathbf{N},\mathbf{S})_{C_0(E^0 \setminus F^0)} \}$.
\end{enumerate}
\end{prop}
\begin{proof}
We prove Equality~(\ref{Theta_{x,y}=Theta_{x_{X^1},y_{X^1}}}). Fix $x, y \in C_c(E,\mathbf{N},\mathbf{S})$, and fix $z \in C_c(F^1,\mathbf{N}^{F^1},\mathbf{S}^{F^1})$. Take $w \in C_c(E,\mathbf{N},\mathbf{S})$ with $w_\alpha \vert_{\overline{N_\alpha \cap F^1}}=z_\alpha$, for all $\alpha \in \Lambda$. Proposition~\ref{omega is a homomorphism} implies 
\begin{align*}
\omega(\Theta_{x,y})(z)=\big(x_\alpha \vert_{\overline{N_\alpha \cap F^1}} \langle y,w \rangle_{C_0(E^0)} \circ s \vert_{\overline{N_\alpha \cap F^1}} \big)=\Theta_{(x_\alpha \vert_{\overline{N_\alpha \cap F^1}}),(y_\alpha \vert_{\overline{N_\alpha \cap F^1}})}z.
\end{align*}

For Equality~(\ref{image of omega on K(X(E,N,S))}), observe that the containment $\subset$ follows immediately from~(1). For the reverse containment, Fix $x, y \in C_c(F^1,\mathbf{N}^{F^1},\mathbf{S}^{F^1})$. Choose $z, w \in C_c(E,\mathbf{N},\mathbf{S})$ with $z_\alpha \vert_{\overline{N_\alpha \cap F^1}}=x_\alpha$, and $w_\alpha \vert_{\overline{N_\alpha \cap F^1}}=y_\alpha$ for all $\alpha \in \Lambda$. By Equality~(\ref{Theta_{x,y}=Theta_{x_{X^1},y_{X^1}}}), $\omega(\Theta_{z,w})=\Theta_{x,y}$.

Finally we prove Equality~(\ref{ker of omega on K(X(E,N,S))}). Fix $K \in \ker(\omega) \cap \mathcal{K}(X(E,\mathbf{N},\mathbf{S}))$. Since $\ker(\omega) \cap \mathcal{K}(X(E,\mathbf{N},\mathbf{S}))$ is a $C^*$-subalgebra of $\mathcal{K}(X(E,\mathbf{N},\mathbf{S}))$, the Hahn-decomposition gives $K=K_1K_1^*-K_2 K_2^*+iK_3 K_3^*-iK_4 K_4^*$, where each $K_i \in \ker(\omega) \cap \mathcal{K}(X(E,\mathbf{N},\mathbf{S}))$. Take a sequence $(E_n) \subset \lsp\{\Theta_{x,y}:x,y \in X(E,\mathbf{N},\mathbf{S})\}$ such that $E_n K_i, K_i E_n \to K_i$ for each $i$. Then 
\[
K_1 E_n K_1^*-K_2 E_n K_2^*+iK_3 E_n K_3^*-i K_4 E_n K_4^* \to K.
\]
Equation~(\ref{compute of the kernel of omega}) gives $K \in \overline{\lsp}\{\Theta_{x,y}: x, y \in X(E,\mathbf{N},\mathbf{S})_{C_0(E^0 \setminus F^0)} \}$. Conversely, fix $x, y \in X(E,\mathbf{N},\mathbf{S})_{C_0(E^0 \setminus F^0)}$. Then $\langle \Theta_{x,y}z,\Theta_{x,y}z \rangle_{C_0(E^0)} \in C_0(E^0 \setminus F^0)$ for all $z \in X(E,\mathbf{N},\mathbf{S})$. By Equation~(\ref{compute of the kernel of omega}), $\Theta_{x,y} \in \ker(\omega)$. Hence Equality~(\ref{ker of omega on K(X(E,N,S))}) holds.
\end{proof}

Next, for an injective covariant Toeplitz representation $(\psi,\pi)$ of $X(E,\mathbf{N},\mathbf{S})$, define $\pi_{0}^{0}:=\pi^{-1}$. We aim to define homomorphisms of the $C^*$-subalgebras $B_{[0,n]} (n \geq 1)$ of the core of $C^*(\psi,\pi)$. The following proposition provides a generalization, and a detailed proof of an assertion made on \cite[Page~4312]{Katsura:TAMS04}.

\begin{prop}\label{define pi_n^n on X(E,N,S)}
Fix an injective covariant Toeplitz representation $(\psi,\pi)$ of $X(E,\mathbf{N},\mathbf{S})$, and fix $n \geq 1$. Then there is a homomorphism $\pi_{n}^{n}:B_{[0,n]} \to \mathcal{L}(X(E_n,\mathbf{N}^n,\mathbf{S}^n))$ such that $\psi_n(\pi_{n}^{n}(b)\xi)=b \psi_n(\xi)$ for all $b \in B_{[0,n]}, \xi \in X(E_n,\mathbf{N}^n,\mathbf{S}^n)$, and such that $\pi_n^n(\psi_n^{(1)}(K))=K$ for all $K \in \mathcal{K}(X(E_n,\mathbf{N}^n,\mathbf{S}^n))$.
\end{prop}
\begin{proof}
We prove the existence of $\pi_{n}^{n}$ by induction on $n \geq 1$. When $n=1, \pi_1^1(\pi(f)+\psi^{(1)}(K)):=\phi(f) +K$ does the job. Suppose that the proposition is true for $n \geq 1$. For $b \in B_{[0,n]}, b_{n+1} \in B_{n+1}$, and for $x_1, \dots, x_{n+1} \in C_c(E,\mathbf{N},\mathbf{S})$, take $(y_m) \subset C_c(E_n,\mathbf{N}^n,\mathbf{S}^n)$ such that $y_m \to \pi_n^n(b)(x_1 \diamond\dots\diamond x_n)$. Define $\pi_n^n(b) \otimes \id \in \mathcal{L}(X(E_{n+1},\mathbf{N}^{n+1},\mathbf{S}^{n+1}))$ by $\pi_n^n(b) \otimes \id (x_1 \diamond\dots\diamond x_{n+1}):=\lim_{m}(y_m \diamond x_{n+1})$. Define $\pi_{n+1}^{n+1}(b+b_{n+1}):=\pi_n^n(b) \otimes \id+(\psi_{n+1}^{(1)})^{-1}(b_{n+1})$. Straightforward calculation shows that $\pi_{n+1}^{n+1}$ is a homomorphism. For $b \in B_{[0,n]}, b_{n+1} \in B_{n+1}, x_1,\dots, x_{n+1} \in C_c(E,\mathbf{N},\mathbf{S})$, we have
\begin{align*}
\psi_{n+1}(\pi_{n+1}^{n+1}&(b+b_{n+1})(x_1\diamond\dots\diamond x_{n+1}))
\\&=\psi_n(\pi_n^n(b)(x_1\diamond\dots\diamond x_n))\psi(x_{n+1})+\psi_{n+1}((\psi_{n+1}^{(1)})^{-1}(b_{n+1})(x_1 \diamond\dots\diamond x_{n+1}))
\\&=(b+b_{n+1})\psi_{n+1}(x_1\diamond\dots\diamond x_{n+1}).
\end{align*}
For $K \in \mathcal{K}(X(E_{n+1},\mathbf{N}^{n+1},\mathbf{S}^{n+1}))$, we have $\pi_{n+1}^{n+1}(\psi_{n+1}^{(1)}(K))=K$ by definition of $\pi_{n+1}^{n+1}$.
\end{proof}

For $n \geq 1$, define a closed subset $E_{\mathrm{sg}}^n:=(s^n)^{-1}(E_{\mathrm{sg}}^0)$ of $E^n$, let $\omega:\mathcal{L}(X(E_n,\mathbf{N}^n,\mathbf{S}^n)) \to \mathcal{L}(X(E_{\mathrm{sg}}^n, (\mathbf{N}^n)^{E_{\mathrm{sg}}^n}, (\mathbf{S}^n)^{E_{\mathrm{sg}}^n}))$ be the homomorphism obtained from Proposition~\ref{omega is a homomorphism}. The following proposition is a generalization of \cite[Lemma~5.1]{Katsura:TAMS04}.

\begin{prop}\label{define pi_{k}^{n} when 0 leq k <n}
Fix an injective covariant Toeplitz representation $(\psi,\pi)$ of $X(E,\mathbf{N},\mathbf{S})$. For $n \geq 1$, there is a homomorphism $\pi_{0}^{n}:B_{[0,n]} \to C_0(E_{\mathrm{sg}}^0)$ such that $\pi_{0}^{n}(b_0+b)=\pi^{-1}(b_0) \vert_{E_{\mathrm{sg}}^0}$ for all $b_0 \in B_0, b \in B_{[1,n]}$. There is a homomorphism $\pi_{0}^{\infty}:B_{[0,\infty]} \to C_0(E_{\mathrm{sg}}^0)$, such that $\pi_{0}^{\infty}(b_0+\dots+b_n)=\pi^{-1}(b_0)\vert_{E_{\mathrm{sg}}^0}$, for all $n \geq 0$, and for all $b_0+\dots + b_n \in B_{[0,n]}$. For $n \geq 2$ and $1 \leq k \leq n-1$, there is a homomorphism $\pi_{k}^{n}:B_{[0,n]} \to \mathcal{L}(X(E_{\mathrm{sg}}^k, (\mathbf{N}^k)^{E_{\mathrm{sg}}^k}, (\mathbf{S}^k)^{E_{\mathrm{sg}}^k}))$ such that $\pi_{k}^{n}(b+c)=\omega\circ\pi_{k}^{k}(b)$ for all $b \in B_{[0,k]}, c \in B_{[k+1,n]}$.
\end{prop}
\begin{proof}
First of all, we claim that for $n \geq 1$, if $b_0 \in B_0$ and $b \in B_{[1,n]}$ satisfying $b_0+b=0$, then $b_0 \in \pi(C_0(E_{\mathrm{rg}}^0))$. We prove this claim by induction on $n \geq 1$. When $n=1$. For $f \in C_0(E^0)$ and $K \in \mathcal{K}(X(E,\mathbf{N},\mathbf{S}))$, if $\pi(f)+\psi^{(1)}(K)=0$, then $f \in C_0(E_{\mathrm{rg}}^0)$. Suppose the claim is true for $n \geq 1$. Then for $b_0 \in B_0, \dots, b_{n+1} \in B_{n+1}$, if $b_0+\dots+b_{n+1}=0$, then by \cite[Proposition~5.12]{Katsura:JFA04} $b_{n+1} \in B_n$. By the induction assumption, $b_0 \in \pi(C_0(E_{\mathrm{rg}}^0))$. So we finish the proof of the claim. For each $n \geq 1$, straightforward calculation shows that there is a well-defined homomorphism $\pi_{0}^{n}$ satisfying the desired formula.

Since the core $B_{[0,\infty]}$ is the direct limit of the increasing sequence $\{B_{[0,n]}\}_{n=1}^{\infty}$, we obtain the homomorphism $\pi_0^\infty$.

We claim that for $n \geq 2, 1 \leq k \leq n-1, b \in B_{[0,k]}$, and $c \in B_{[k+1,n]}$, if $b+c=0$, then $b=\psi_k^{(1)}(K)$, for some $K \in \overline{\lsp}\{\Theta_{\xi,\eta}:\xi,\eta \in X(E_k,\mathbf{N}^k,\mathbf{S}^k)_{C_0(E_{\mathrm{rg}}^0)} \}$. We prove this claim by induction on $n \geq 2$. When $n=2, k=1$. For $b_0+b_1+b_2 \in B_{[0,2]}$, if $b_0+b_1+b_2=0$, then by \cite[Proposition~5.12]{Katsura:JFA04} $b_0+b_1=\psi^{(1)}(K)$, for some $K \in \overline{\lsp}\{\Theta_{\xi,\eta}:\xi,\eta \in X(E,\mathbf{N},\mathbf{S})_{C_0(E_{\mathrm{rg}}^0)} \}$. Suppose that the claim is true for $n \geq 2$. For $b_0+\dots+b_{n+1} \in B_{[0,n+1]}$, if $b_0+\dots+b_{n+1}=0$, then by \cite[Proposition~5.12]{Katsura:JFA04} $b_0+\dots+b_n=-b_{n+1}=\psi_n^{(1)}(K)$, for some $K \in \overline{\lsp}\{\Theta_{\xi,\eta}:\xi,\eta \in X(E_n,\mathbf{N}^n,\mathbf{S}^n)_{C_0(E_{\mathrm{rg}}^0)} \}$. For $1 \leq k \leq n-1$, by the induction, $b_0+\dots+b_k=\psi_k^{(1)}(K')$, for some $K' \in \overline{\lsp}\{\Theta_{\xi,\eta}:\xi,\eta \in X(E_k,\mathbf{N}^k,\mathbf{S}^k)_{C_0(E_{\mathrm{rg}}^0)} \}$. So we finish the proof of the claim. Now fix $n \geq 2, 1 \leq k \leq n-1, b \in B_{[0,k]}$, and $c \in B_{[k+1,n]}$ with $b+c=0$. By the claim there exists $K \in \overline{\lsp}\{\Theta_{\xi,\eta}:\xi,\eta \in X(E_k,\mathbf{N}^k,\mathbf{S}^k)_{C_0(E_{\mathrm{rg}}^0)} \}$ such that $b=\psi_k^{(1)}(K)$. By Proposition~\ref{define pi_n^n on X(E,N,S)}, $\pi_{k}^{k}(b)=K$. By Proposition~\ref{omega maps K(X(E,N,S)) onto K(X(X^1,N^X^1,S^X^1))}, $\omega(\pi_k^k(b))=0$. Hence straightforward calculation shows that there is a well-defined homomorphism $\pi_{k}^{n}$ satisfying the desired formula.
\end{proof}

\begin{cor}\label{direct sum of pi_{k}^{n} from 0 to n-1, direct sum of pi_{k}^{n} from 0 to n}
Fix an injective covariant Toeplitz representation $(\psi,\pi)$ of $X(E,\mathbf{N},\mathbf{S})$. For $n \geq 1$, we have $\bigcap_{k=0}^{n-1}\ker(\pi_{k}^{n})=B_n$. For $n \geq 0$, we have $\bigcap_{k=0}^{n}\ker(\pi_{k}^{n})=\{0\}$.
\end{cor}
\begin{proof}
We prove the first statement by induction on $n \geq 1$. When $n=1$. For $f \in C_0(E^0)$ and $K \in \mathcal{K}(X(E,\mathbf{N},\mathbf{S}))$, if $\pi(f)+\psi^{(1)}(K) \in \ker(\pi_0^1)$ then by Proposition~\ref{define pi_{k}^{n} when 0 leq k <n} $f(E_{\mathrm{sg}}^0)=0$. The covariance of $(\psi,\pi)$ implies that $\pi(f)+\psi^{(1)}(K)=\psi^{(1)}(\phi(f)+K)$. Conversely, for $K \in \mathcal{K}(X(E,\mathbf{N},\mathbf{S}))$, by Proposition~\ref{define pi_{k}^{n} when 0 leq k <n} $\pi_0^1(\psi^{(1)}(K))=0$. So $\ker(\pi_0^1)=B_1$. Suppose that $\bigcap_{k=0}^{n-1}\ker(\pi_{k}^{n})=B_n$, for some $n \geq 1$. We show that $\bigcap_{k=0}^{n}\ker(\pi_{k}^{n+1})=B_{n+1}$. For $b_0+\dots+b_{n+1} \in \bigcap_{k=0}^{n}\ker(\pi_{k}^{n+1})$, by Proposition~\ref{define pi_{k}^{n} when 0 leq k <n}, $b_0+\dots+b_n \in \bigcap_{k=0}^{n-1}\ker(\pi_{k}^{n})$. By the induction assumption, there exists $K \in \mathcal{K}(X(E_n,\mathbf{N}^n,\mathbf{S}^n))$ such that $b_0+\dots+b_n=\psi_n^{(1)}(K)$. By Proposition~\ref{define pi_n^n on X(E,N,S)}, 
\[
\pi_n^{n+1}(\psi_n^{(1)}(K)+b_{n+1})=\omega\circ\pi_n^n(\psi_n^{(1)}(K))=\omega(K)=0.
\]
By Proposition~\ref{omega maps K(X(E,N,S)) onto K(X(X^1,N^X^1,S^X^1))} and by \cite[Proposition~5.12]{Katsura:JFA04}, $\psi_n^{(1)}(K) \in B_{n+1}$. So $\bigcap_{k=0}^{n}\ker(\pi_{k}^{n+1}) \subset B_{n+1}$. By definition of $\pi_{k}^{n+1} (k \leq n)$ we clearly have $B_{n+1} \subset \bigcap_{k=0}^{n}\ker(\pi_{k}^{n+1})$. So $B_{n+1}=\bigcap_{k=0}^{n}\ker(\pi_{k}^{n+1})$.

The second statement is trivial for $n=0$. When $n \geq 1, \bigcap_{k=0}^{n}\ker(\pi_{k}^{n})=B_n \cap \ker(\pi_{n}^{n})$. For $\psi_n^{(1)}(K) \in B_n \cap \ker(\pi_{n}^{n})$, by Proposition~\ref{define pi_n^n on X(E,N,S)}, $\pi_{n}^{n}(\psi_n^{(1)}(K))=K=0$. Hence $\bigcap_{k=0}^{n}\ker(\pi_{k}^{n})=\{0\}$.
\end{proof}

We finish the section by constructing a covariant Toeplitz representation (of a modified topological graph $E_Y$) from a non-covariant Toeplitz representation of $E$ under certain condition. This technique is a generalization of Katsura's work in \cite[Section~3]{Katsura:IJM06}. Fix a closed subset $Y$ of $E_{\mathrm{rg}}^0$ in the subspace topology of $E_{\mathrm{rg}}^0$ and define $\partial Y:=\overline{Y} \setminus Y$. 

\begin{defn}[{\cite[Page~799]{Katsura:IJM06}}]\label{define E_Y where Y is a closed subset of E_rg^0}
Define a topological graph $E_Y:=(E_Y^0,E_Y^1,r_Y,s_Y)$ as follows:
\begin{enumerate}
\item $E_Y^0:=E^0 \amalg_{\partial Y}\overline{Y}=\{(v,0), (w,1):v \in E^0, w \in \overline{Y}\}$;
\item $E_Y^1:=E^1 \amalg_{s^{-1}(\partial Y)}s^{-1}(\overline{Y})=\{(e,0), (f,1):e \in E^1, f \in s^{-1}(\overline{Y}) \}$;
\item $r_Y(e,n):=(r(e),0)$, for all $(e,n) \in E_Y^1$; and
\item $s_Y(e,n):=(s(e),n)$, for all $(e,n) \in E_Y^1$.
\end{enumerate}
Define a cover $\mathbf{N}_Y:=\big\{\big(N_\alpha \times \{0\}\big) \cup \big((N_\alpha \cap s^{-1}(\overline{Y})) \times \{1\}\big) \big\}_{\alpha \in \Lambda}$ of $E_Y^1$ by precompact open $s_Y$-sections. Define a $1$-cocycle $\mathbf{S}_Y:=\{s_{\alpha\beta, Y}\}_{\alpha, \beta \in \Lambda}$ relative to $\mathbf{N}_Y$ by $s_{\alpha\beta,Y}(e,n):=s_{\alpha\beta}(e)$ for all $(e,n) \in (\overline{N_{\alpha\beta}} \times \{0\}) \cup (\overline{N_{\alpha\beta} \cap s^{-1}(\overline{Y})} \times \{1\})$.
\end{defn}

Let $p_Y^0: E_Y^0 \to E^0$ and $p_Y^1:E_Y^1 \to E^1$ be the two surjective proper continuous projections. Let $(p_Y^0)_*:C_0(E^0) \to C_0(E_Y^0)$ be injective homomorphism obtained from $p_Y^0$, and let $(p_Y^1)_*:X(E,\mathbf{N},\mathbf{S}) \to X(E_Y,\mathbf{N}_Y,\mathbf{S}_Y)$ be the norm-preserving linear map obtained from $p_Y^1$. Let $t_Y=(t_Y^0,t_Y^1)$ be the injective universal covariant Toeplitz representation of $X(E_Y,\mathbf{N}_Y,\mathbf{S}_Y)$ in $\mathcal{O}(E_Y,\mathbf{N}_Y,\mathbf{S}_Y)$. Fix a Toeplitz representation $(\psi,\pi)$ of $X(E,\mathbf{N},\mathbf{S})$ in a $C^*$-algebra $B$ such that $C_0(E_{\mathrm{rg}}^0 \setminus Y) \subset \{f \in C_0(E_{\mathrm{rg}}^0 ): \pi(f)=\psi^{(1)}(\phi(f))\}$.

The following lemma is a generalization of \cite[Page~801, Lemma~3.9]{Katsura:IJM06}.

\begin{lemma}\label{define psi_rg,pi_rg}
There is a linear map $\psi_{\mathrm{rg}}:X(E,\mathbf{N},\mathbf{S})_{C_0(E_{\mathrm{rg}}^0)} \to B$ such that $\psi_{\mathrm{rg}}(x\cdot f)=\psi(x)\psi^{(1)}(\phi(f))$ for all $x \in X(E,\mathbf{N},\mathbf{S})$ and $f \in C_0(E_{\mathrm{rg}}^0)$. There is a homomorphism $\pi_{\mathrm{rg}}:C_0(E_{\mathrm{rg}}^0) \to B$ such that $\pi_{\mathrm{rg}}(f)=\psi^{(1)}(\phi(f))$ for all $f \in C_0(E_{\mathrm{rg}}^0)$.
\end{lemma}
\begin{proof}
Straightforward calculation yields the results.
\end{proof}

\begin{lemma}\label{decompose function in X amalg_partial Y Y}
Let $X$ be a locally compact Hausdorff space, let $U$ be an open set in $X$, and let $Y$ be a closed subset of $U$ in the subspace topology of $U$. Then for any $f \in C_0(X \amalg_{\overline{Y} \setminus Y} \overline{Y})$, there exist $g \in C_0(U)$ and $h \in C_0(X)$, such that $f(x,0)=g(x)+h(x)$ for all $x \in X$, and $f(y,1)=h(y)$ for all $y \in \overline{Y}$.
\end{lemma}
\begin{proof}
It follows from the proof of \cite[Lemma~3.6]{Katsura:IJM06}.
\end{proof}

The following proposition is a generalization of \cite[Proposition~3.15]{Katsura:IJM06}.

\begin{prop}\label{tilde{psi},tilde{pi} is a cov Toep rep}
There is a covariant Toeplitz representation $(\widetilde{\psi},\widetilde{\pi})$ of $X(E_Y,\mathbf{N}_Y,\mathbf{S}_Y)$ in $C^*(\psi,\pi)$ such that
\begin{enumerate}
\item\label{define widetilde{psi}} $\widetilde{\pi}(f)=\pi_{\mathrm{rg}}(g)+\pi(h)$, where $f \in C_0(E_Y^0), g \in C_0(E_{\mathrm{rg}}^0),h \in C_0(E^0)$, such that $f(v,0)=g(v)+h(v)$ for all $v \in E^0$, and $f(v,1)=h(v)$ for all $v \in Y$;
\item\label{define widetilde{pi}} $\widetilde{\psi}(x)=\psi_{\mathrm{rg}}(y)+\psi(z)$, where $x \in C_c(E_Y,\mathbf{N}_Y,\mathbf{S}_Y)$, $y \in X(E,\mathbf{N},\mathbf{S})_{C_0(E_{\mathrm{rg}}^0 )} \cap C_c(E,\mathbf{N},\mathbf{S}), z \in C_c(E,\mathbf{N},\mathbf{S})$ such that $x_\alpha(e,0)=y_\alpha(e)+z_\alpha(e)$ for all $e \in \overline{N_\alpha}$, and $x_\alpha(e,1)=z_\alpha(e)$ for all $e \in \overline{N_\alpha} \cap s^{-1}(Y)$;
\item\label{widetilde{T}^i circ (p_Y^i)^*=T^i}$\widetilde{\pi} \circ (p_Y^0)_*=\pi, \widetilde{\psi} \circ (p_Y^1)_*=\psi$; and
\item\label{C^*(widetilde{psi},widetilde{pi})=C^*(psi,pi)} $C^*(\widetilde{\psi},\widetilde{\pi})=C^*(\psi,\pi)$.
\end{enumerate}
Moreover, $\widetilde{\pi}$ is injective if and only if $\pi$ is injective and $C_0(E_{\mathrm{rg}}^0 \setminus Y)=\{f \in C_0(E_{\mathrm{rg}}^0 ): \pi(f)=\psi^{(1)}(\phi(f))\}$.
\end{prop}
\begin{proof}
Fix $f \in C_0(E_Y^0)$. Lemma~\ref{decompose function in X amalg_partial Y Y} yields $g \in C_0(E_{\mathrm{rg}}^0),h \in C_0(E^0)$ such that $f(v,0)=g(v)+h(v)$ for all $v \in E^0$, and $f(v,1)=h(v)$ for all $v \in \overline{Y}$. For $g' \in C_0(E_{\mathrm{rg}}^0),h' \in C_0(E^0)$ such that $f(v,0)=g'(v)+h'(v)$ for all $v \in E^0$, and $f(v,1)=h'(v)$ for all $v \in \overline{Y}$, we have $h'-h=g-g' \in C_0(E_{\mathrm{rg}}^0 \setminus Y)$. So $\pi(h'-h)=\pi(g-g')=\psi^{(1)}(\phi(g-g'))$. Define $\widetilde{\pi}:C_0(E_Y^0) \to C^*(\psi,\pi)$ by $\widetilde{\pi}(f):=\pi_{\mathrm{rg}}(g)+\pi(h)$. It is straightforward to check that $\widetilde{\pi}$ is a homomorphism.

Now fix $f \in C_c\big(\big(N_{\alpha_0} \times \{0\}\big) \cup \big((N_{\alpha_0} \cap s^{-1}(\overline{Y})) \times \{1\}\big)\big)$. Lemma~\ref{decompose function in X amalg_partial Y Y} yields $g \in C_0(s^{-1}(E_{\mathrm{rg}}^0)),h \in C_0(E^1)$ such that $f(e,0)=g(e)+h(e)$ for all $e \in E^1$, and $f(e,1)=h(e)$ for all $e \in s^{-1}(\overline{Y})$. Since $K:=p_Y^1(\supp(f)) \subset N_{\alpha_0}$, the Urysohn's lemma gives $l \in C_0(N_{\alpha_0})$ such that $l(K)=1$. Then $f(e,0)=l(e)g(e)+l(e)h(e)$ for all $e \in E^1$, and $f(e,1)=l(e)h(e)$ for all $e \in s^{-1}(\overline{Y})$. Let $x:=(f^{\Ind_{\alpha_0}^{\alpha}})$, let $y:=((lg)^{\Ind_{\alpha_0}^{\alpha}})$, and let $z:=((lh)^{\Ind_{\alpha_0}^{\alpha}})$. Then $y \in X(E,\mathbf{N},\mathbf{S})_{C_0(E_{\mathrm{rg}}^0 )}, x_\alpha(e,0)=y_\alpha(e)+z_\alpha(e)$ for all $e \in \overline{N_\alpha}$, and $x_\alpha(e,1)=z_\alpha(e)$ for all $e \in \overline{N_\alpha \cap s^{-1}(\overline{Y})}$. Proposition~\ref{induced tuple generates C_c tuple} implies that for $x \in C_c(E_Y,\mathbf{N}_Y,\mathbf{S}_Y)$, there exist $y \in X(E,\mathbf{N},\mathbf{S})_{C_0(E_{\mathrm{rg}}^0 )} \cap C_c(E,\mathbf{N},\mathbf{S}),z \in C_c(E,\mathbf{N},\mathbf{S})$, such that $x_\alpha(e,0)=y_\alpha(e)+z_\alpha(e)$ for all $e \in \overline{N_\alpha}$, and $x_\alpha(e,1)=z_\alpha(e)$ for all $e \in \overline{N_\alpha \cap s^{-1}(\overline{Y})}$. By the similar argument in the previous paragraph, there is a bounded linear map $\widetilde{\psi}:X(E_Y,\mathbf{N}_Y,\mathbf{S}_Y) \to C^*(\psi,\pi)$ by $\widetilde{\psi}(x):=\psi_{\mathrm{rg}}(z)+\psi(z)$.

It is straightforward to check that $(\widetilde{\psi},\widetilde{\pi})$ is a Toeplitz representation of $X(E_Y,\mathbf{N}_Y,\mathbf{S}_Y)$ in $C^*(\psi,\pi)$.

Now we check Equality~(\ref{widetilde{T}^i circ (p_Y^i)^*=T^i}). For $f \in C_0(E^0)$, we have $(p_Y^0)_*(f)(v,0)=f(v)$ for all $v \in E^0$, and $(p_Y^0)_*(f)(v,1)=f(v)$ for all $v \in \overline{Y}$. So $\widetilde{\pi} \circ (p_Y^0)_*(f)=\pi(f)$ by definition of $\widetilde{\pi}$. Hence $\widetilde{\pi} \circ (p_Y^0)_*=\pi$. Similarly, $\widetilde{\psi} \circ (p_Y^1)_*=\psi$. Equality~(\ref{C^*(widetilde{psi},widetilde{pi})=C^*(psi,pi)}) follows easily from Equality~(\ref{widetilde{T}^i circ (p_Y^i)^*=T^i}).

Next we prove the covariance of $(\widetilde{\psi},\widetilde{\pi})$. By Proposition~\ref{phi^{-1}(K(X)) intersects ker(phi^{perp}) of X_{E,N,S}} and \cite[Lemma~3.3]{Katsura:IJM06}, $J_{X(E_Y,\mathbf{N}_Y \mathbf{S}_Y)}=C_0(E_{\mathrm{rg}}^0 \times \{0\})$. Fix a nonnegative function $f \in C_c(E_{\mathrm{rg}}^0 \times \{0\} )$. Lemma~\ref{decompose function in X amalg_partial Y Y} gives $g \in C_0(E_{\mathrm{rg}}^0), h \in C_0(E_{\mathrm{rg}}^0 \setminus Y)$ such that $f(v,0)=g(v)+h(v)$ for all $v \in E^0$. Then there exists a finite open cover $\{N_{\alpha_i} \}_{i=1}^{n}$ of $r^{-1}(\supp(g+h))$. We use a partition of unity to get a finite collection of functions $\{h_i\}_{i=1}^{n} \subset C(E^1,[0,1])$ such that $\supp(h_i) \subset N_{\alpha_i}$ for all $i$, and $\sum_{i=1}^{n}h_i=1$ on $r^{-1}(\supp(g+h))$. By Equation~(\ref{computation of phi(f) in the twisted graph correspondence}), we have 
\[
\phi(g+h)=\sum_{i=1}^{n}\Theta_{\sqrt{ h_i (g+h) \circ r}^{\mathrm{Ind}_{\alpha_i}^{*}},\sqrt{ h_i (g+h) \circ r}^{\mathrm{Ind}_{\alpha_i}^{*}}},
\]
and
\[
\phi(f)=\sum_{i=1}^{n}\Theta_{(p_Y^1)_*\Big(\sqrt{ (h_i(g+h) \circ r}^{\mathrm{Ind}_{\alpha_i}^{*}}\Big),(p_Y^1)_*\Big(\sqrt{h_i (g+h)\circ r}^{\mathrm{Ind}_{\alpha_i}^{*}}\Big)}.
\]
So
\begin{align*}
\widetilde{\psi}^{(1)}(\phi(f))&=\sum_{i=1}^{n}\psi\Big(\sqrt{ h_i (g+h) \circ r}^{\mathrm{Ind}_{\alpha_i}^{*}}\Big)\psi\Big(\sqrt{ h_i (g+h) \circ r}^{\mathrm{Ind}_{\alpha_i}^{*}}\Big)^*
\\&=\psi^{(1)}(\phi(g))+\psi^{(1)}(\phi(h))=\pi_{\mathrm{rg}}(g)+\pi(h)=\widetilde{\pi}(f).
\end{align*}
Hence $(\widetilde{\psi},\widetilde{\pi})$ is covariant.

We prove the last statement. Suppose that $\widetilde{\pi}$ is injective. For $f \in C_0(E^0)$, if $\pi(f)=0$ then by Equality~\ref{widetilde{T}^i circ (p_Y^i)^*=T^i} $(p_Y^0)_*(f)=0$. So $f=0$ and $\pi$ is injective. Fix $f \in C_c(E_\mathrm{rg}^0)$ such that $\psi^{(1)}(\phi(f))=\pi(f)$. By the Urysohn's lemma, there exists $g \in C_0(E_\mathrm{rg}^0 \times \{0\})$ such that $g(\supp(f)\times \{0\})=1$. Then $\widetilde{\pi}(g(p_Y^0)_*(f))=\psi^{(1)}(\phi(f))$, and $\widetilde{\pi}((p_Y^0)_*(f))=\pi(f)$. So $g(p_Y^0)_*(f)=(p_Y^0)_*(f)$ since $\widetilde{\pi}$ is injective. Hence $f(Y)=0$ and $C_0(E_{\mathrm{rg}}^0 \setminus Y)=\{f \in C_0(E_{\mathrm{rg}}^0 ): \pi(f)=\psi^{(1)}(\phi(f))\}$. Conversely, suppose that $\pi$ is injective and $C_0(E_{\mathrm{rg}}^0 \setminus Y)=\{f \in C_0(E_{\mathrm{rg}}^0 ): \pi(f)=\psi^{(1)}(\phi(f))\}$. Fix $f \in C_0(E_Y^0)$ such that $\widetilde{\pi}(f)=0$. Lemma~\ref{decompose function in X amalg_partial Y Y} yields $g \in C_0(E_{\mathrm{rg}}^0),h \in C_0(E^0)$ such that $f(v,0)=g(v)+h(v)$ for all $v \in E^0$, and $f(v,1)=h(v)$ for all $v \in \overline{Y}$. Then $\pi(h)=\psi^{(1)}(\phi(-g))$ because $\widetilde{\pi}(f)=0$. So $h \in C_0(E_{\mathrm{rg}}^0), \pi(h)=\psi^{(1)}(\phi(h))$. By assumption $h \in C_0(E_{\mathrm{rg}}^0 \setminus Y)$. Hence $f(v,1)=0$ for all $v \in \overline{Y}$. Since $\pi$ is injective, $\phi(g+h)=0$. We get $g+h=0$ and $f=0$. Therefore $\widetilde{\pi}$ is injective.
\end{proof}

\section{The Cuntz-Krieger Uniqueness Theorem}

In this section we prove a version of the Cuntz-Krieger uniqueness theorem for twisted topological graph algebras by following Katsura's idea in \cite{Katsura:TAMS04}. To begin with, we recall the notion of topological freeness from \cite{Katsura:TAMS04}.

\begin{defn}[{\cite[Definitions~5.3--5.5]{Katsura:TAMS04}}]
Let $E$ be a topological graph. For $n \geq 2$, a nonempty subset $S \subset E^n$ is \emph{non-returning} if $e_n \neq e_i'$ for all $i=1, \dots, n-1$, whenever $(e_1,\dots,e_n), (e_1',\dots,e_n') \in S$. For $n \geq 1$, a finite path $(e_1,\dots, e_n) \in E^n$ is a \emph{cycle} if $r(e_1)=s(e_n)$, and the vertices $r(e_1),\dots, r(e_n)$ are the \emph{base points} of the cycle. The cycle is \emph{without entrances} if $r^{-1}(r(e_i))=\{e_i\}$, for $i=1, \dots, n$. The graph $E$ is \emph{topologically free} if the set of base points of cycles without entrances has empty interior.
\end{defn}

\begin{thm}[The Cuntz-Krieger uniqueness theorem]\label{twisted version of Cuntz-Krieger Uniqueness Theorem}
Let $E$ be a topologically free topological graph, let $\mathbf{N}=\{N_\alpha\}_{\alpha \in \Lambda}$ be a cover of $E^1$ by precompact open $s$-sections, let $\mathbf{S}=\{s_{\alpha\beta}\}_{\alpha,\beta \in \Lambda}$ be a $1$-cocycle relative to $\mathbf{N}$, and let $(j_X,j_A)$ be the injective universal covariant Toeplitz representation of $X(E,\mathbf{N},\mathbf{S})$ in $\mathcal{O}(E,\mathbf{N},\mathbf{S})$. Fix an injective covariant Toeplitz representation $(\psi,\pi)$ of $X(E,\mathbf{N},\mathbf{S})$. Let $h:\mathcal{O}(E,\mathbf{N},\mathbf{S}) \to C^*(\psi,\pi)$ be the homomorphism such that $h \circ \psi=j_X, h \circ \pi=j_A$. Then $h$ is an isomorphism.
\end{thm}
\begin{proof}
Fix $L \geq 1$, $n_i, m_i \geq 1, \xi_i \in X(E_{n_i}, \mathbf{N}^{n_i}, \mathbf{S}^{n_i})$, and $\eta_i \in X(E_{m_i}, \mathbf{N}^{m_i}, \mathbf{S}^{m_i}), i=1,\dots, L$. Suppose that $\xi_i=\xi_{i1} \diamond\dots\diamond \xi_{in_i}$, where $\xi_{i1}, \dots, \xi_{in_i} \in C_c(E,\mathbf{N},\mathbf{S})$  whenever $n_i \ge 1$; and similarly for the $\eta_i$. Suppose that if $n=0$ then $\xi_i \in C_c(E^0)$; and if $m=0$ then $\eta_i \in C_c(E^0)$. Let $x:=\sum_{i=1}^{L}\psi_{n_i}(\xi_i)\psi_{m_i}(\eta_i)^*$, and let $x_0:=\sum_{n_i=m_i}\psi_{n_i}(\xi_i)\psi_{m_i}(\eta_i)^*$. An argument like that in the proof of \cite[Theorem~5.12]{Katsura:TAMS04} shows that we only need to show $\Vert x_0 \Vert \leq \Vert x \Vert$.

Let $n:=\max\{n_i,m_i:1\leq i \leq L\}$. If $n=0$ then $x_0=x$ which implies that $\Vert x_0 \Vert \leq \Vert x \Vert$ automatically. We now assume that $n \geq 1$ and $x_0 \neq 0$. To prove that $\Vert x_0 \Vert \leq \Vert x \Vert$, it is enough to verify that for any $\epsilon >0$, there exist $a, b \in C^*(\psi,\pi)$ and $f \in C_0(E^0)$, such that $\Vert a \Vert, \Vert b \Vert \leq 1, \Vert f \Vert=\Vert x_0 \Vert$, and $\Vert a^*xb-\pi(f)\Vert<\epsilon$. So we fix $\epsilon>0$ with $\epsilon <\Vert x_0 \Vert$. 

Propositions~\ref{define pi_n^n on X(E,N,S)}, \ref{define pi_{k}^{n} when 0 leq k <n} yield a homomorphism 
\[
\bigoplus_{k=0}^{n}\pi_k^n:B_{[0,n]} \to C_0(E_{\mathrm{sg}}^0) \oplus \Big(\bigoplus_{k=1}^{n-1}\mathcal{L}(X(E_{\mathrm{sg}}^k, (\mathbf{N}^k)^{E_{\mathrm{sg}}^k}, (\mathbf{S}^k)^{E_{\mathrm{sg}}^k}))\Big) \oplus \mathcal{L}(X(E_n,\mathbf{N}^n,\mathbf{S}^n))
\]
By Corollary~\ref{direct sum of pi_{k}^{n} from 0 to n-1, direct sum of pi_{k}^{n} from 0 to n}, $\bigoplus_{k=0}^{n}\pi_k^n$ is injective, so there exists $0\leq k \leq n$, such that $\Vert \pi_{k}^{n}(x_0) \Vert=\Vert x_0\Vert$. We consider three cases: $k = 0$, $1 \leq k \leq n-1$, and $k = n$.

Case $1$: Suppose that $1 \leq k \leq n-1$. Take $\xi, \eta \in C_c(E_{\mathrm{sg}}^k, (\mathbf{N}^k)^{E_{\mathrm{sg}}^k}, (\mathbf{S}^k)^{E_{\mathrm{sg}}^k})$ with $\Vert \xi \Vert_{C_0(E_{\mathrm{sg}}^0)}=\Vert \eta \Vert_{C_0(E_{\mathrm{sg}}^0)}=1$, such that $\Vert x_0\Vert-\epsilon /2 < \Vert\langle \xi,\pi_{k}^{n}(x_0)\eta \rangle_{C_0(E_{\mathrm{sg}}^0)}\Vert\leq\Vert x_0 \Vert$. By Proposition~\ref{Katsura's Ext Thm}, there exist $\widetilde{\xi}, \widetilde{\eta} \in C_c(E_k,\mathbf{N}^k,\mathbf{S}^k)$ with $\Vert \widetilde{\xi}\Vert_{C_0(E^0)}=\Vert \widetilde{\eta}\Vert_{C_0(E^0)}=1$, such that $\widetilde{\xi}_{\alpha_1,\dots,\alpha_k}\vert_{\overline{N_{\alpha_1}\times\dots\times N_{\alpha_k} \cap E_{\mathrm{sg}}^k}}=\xi_{\alpha_1,\dots,\alpha_k}$ and $\widetilde{\eta}_{\alpha_1,\dots,\alpha_k}\vert_{\overline{N_{\alpha_1}\times\dots\times N_{\alpha_k} \cap E_{\mathrm{sg}}^k}}=\eta_{\alpha_1,\dots,\alpha_k}$. By Proposition~\ref{define pi_{k}^{n} when 0 leq k <n},
\begin{align*}
\langle \xi,\pi_{k}^{n}(x_0)\eta \rangle_{C_0(E_{\mathrm{sg}}^0)}&=\Big\langle \xi,\omega\circ \pi_k^k\Big(\sum_{n_i=m_i \leq k}\psi_{n_i}(\xi_i)\psi_{m_i}(\eta_i)^*\Big)(\eta)\Big\rangle_{C_0(E_{\mathrm{sg}}^0)}
\\&=\Big\langle \widetilde{\xi},\pi_k^k\Big(\sum_{n_i=m_i \leq k}\psi_{n_i}(\xi_i)\psi_{m_i}(\eta_i)^*\Big)(\widetilde{\eta})\Big\rangle_{C_0(E^0)}\Big\vert_{E_{\mathrm{sg}}^0}.
\end{align*}
Define $g:=\Big\langle \widetilde{\xi},\pi_k^k\Big(\sum_{n_i=m_i \leq k}\psi_{n_i}(\xi_i)\psi_{m_i}(\eta_i)^*\Big)(\widetilde{\eta})\Big\rangle_{C_0(E^0)}$. By Proposition~\ref{define pi_n^n on X(E,N,S)}, we have
\begin{align*}
\pi(g)&=\psi_k(\widetilde{\xi})^*\psi_k\Big(\pi_k^k\Big(\sum_{n_i=m_i \leq k}\psi_{n_i}(\xi_i)\psi_{m_i}(\eta_i)^*\Big)(\widetilde{\eta})\Big)
\\&=\sum_{n_i=m_i\leq k}\psi_k(\widetilde{\xi})^*\psi_{n_i}(\xi_i)\psi_{m_i}(\eta_i)^*\psi_k(\widetilde{\eta}). 
\end{align*}
Hence $\Vert x_0 \Vert-\epsilon /2< \Vert g \vert_{E_{\mathrm{sg}}^0}\Vert  \leq \Vert x_0 \Vert$. Therefore there exists $v \in E_{\mathrm{sg}}^0$, such that $\Vert x_0 \Vert-\epsilon /2< \vert g(v)\vert \leq \Vert x_0 \Vert$. By definition of $E_{\mathrm{sg}}^0$ we now split into two subcases.

Subcase $1.1$: Suppose that $v \in \overline{E^0 \setminus \overline{r(E^1)}}$. By continuity of $g$, there exists $v' \in E^0 \setminus \overline{r(E^1)}$, such that $\Vert x_0 \Vert-\epsilon /2< \vert g(v')\vert \leq \Vert x_0 \Vert$. By the Urysohn's lemma, there exists $h \in C_0(E^0,[0,1])$, such that $h(\overline{r(E^1)})=0, h(v')=1$. Let $a:=\psi_k(\widetilde{\xi})\pi(h)$, let $b:=\psi_k(\widetilde{\eta})\pi(h)$, and let $f:=(\Vert x_0 \Vert / \Vert hgh \Vert) hgh$. If $n_i=m_i >k$ or $n_i \neq m_i$, we deduce that $a^*\psi_{n_i}(\xi_i)\psi_{m_i}(\eta_i)^*b=0$ because $h \cdot y = 0$ for all $y \in X(E,\mathbf{N},\mathbf{S})$. So
\begin{align*}
\Vert a^*xb-\pi(f) \Vert=\Vert \pi(h)\pi(g)\pi(h)-\pi(f)\Vert<\epsilon.
\end{align*}

Subcase $1.2$: Suppose that $v \in E^0 \setminus E_{\mathrm{fin}}^0$. By continuity of $g$, there exists an open neighborhood $V$ of $v$, such that $\vert g(w)-g(w')\vert<\epsilon /2$ for all $w, w' \in V$. Define a compact subset of $E^1$
\begin{align*}
K:&=\Big( \bigcup_{n_i \geq 1}\{e_1,\dots,e_{n_i}: (e_1,\dots,e_{n_i}) \in \supp([\xi_i \vert \xi_i]) \}\Big) \bigcup 
\\&\ \ \ \  \Big( \bigcup_{m_i \geq 1}\{e_1,\dots,e_{m_i}: (e_1,\dots,e_{m_i}) \in \supp([\eta_i \vert \eta_i]) \}\Big) \bigcup
\\&\ \ \ \ \{e_1,\dots,e_k:(e_1,\dots,e_k) \in \supp([\widetilde{\xi}\vert\widetilde{\xi}])\cup \supp([\widetilde{\eta}\vert\widetilde{\eta}])\}.
\end{align*}
By definition of $E_{\mathrm{fin}}^0$, we have $r^{-1}(V) \not\subset K$. Take $e \in N_{\alpha_0} \cap r^{-1}(V) \setminus K$. By the Urysohn's lemma there exists $h \in C_0(N_{\alpha_0},[0,1])$, such that $h(e)=1$, and $h(K \cup r^{-1}(V^c))=0$. Define $y:=(h^{\mathrm{Ind}_{\alpha_0}^{\alpha}})_{\alpha \in \Lambda}$. Let $a:=\psi_k(\widetilde{\xi})\psi(y)$, let $b:=\psi_k(\widetilde{\eta})\psi(y)$, and let $f:=(\Vert x_0 \Vert/ \Vert \langle y,g \cdot y \rangle_{C_0(E^0)}\Vert)\langle y,g \cdot y \rangle_{C_0(E^0)}$. Since $\langle y,z \rangle_{C_0(E^0)}=0$ for all $z \in C_c(E,\mathbf{N},\mathbf{S})$ satisfying $\supp([z \vert z]) \subset K$, if $n_i=m_i >k$ or $n_i \neq m_i$, then $a^*\psi_{n_i}(\xi_i)\psi_{m_i}(\eta_i)^*b=0$. So
\begin{align*}
\Vert a^*xb-\pi(f) \Vert=\Vert \psi(y)^*\pi(g)\psi(y)-\pi(f)\Vert=\Vert \langle y,g \cdot y \rangle_{C_0(E^0)}-f \Vert<\epsilon.
\end{align*}
This completes the case for $1 \leq k \leq n-1$.

Case $2$: Suppose that $k=0$. The argument for this case is similar to the Case $1$.

Case $3$: Suppose that $k=n$. Take $\xi, \eta \in C_c(E_n, \mathbf{N}^n,\mathbf{S}^n)$ with $\Vert \xi \Vert_{C_0(E_0)}=\Vert \eta \Vert_{C_0(E^0)}=1$, such that $\Vert x_0\Vert-\epsilon /2 < \Vert\langle \xi,\pi_{n}^{n}(x_0)\eta \rangle_{C_0(E_0)}\Vert\leq\Vert x_0 \Vert$. Define $g:=\langle \xi,\pi_{n}^{n}(x_0)\eta \rangle_{C_0(E_0)}$. By Proposition~\ref{define pi_n^n on X(E,N,S)}, $\psi_n(\widetilde{\xi})^*x_0 \psi_n(\widetilde{\eta})=\pi(g)$. Since $\Vert x_0\Vert-\epsilon /2 < \Vert g \Vert \leq\Vert x_0 \Vert$, by continuity of $g$, there exist $v \in E^0$ and an open neighborhood $V$ of $v$, such that $\Vert x_0\Vert-\epsilon /2 < \vert g(w) \vert \leq\Vert x_0 \Vert$, for all $w \in V$. 

Subcase $3.1$: Suppose that there exists $1 \leq M \leq n$ such that $(r^M)^{-1}(V) \neq \emptyset$ and $(r^{M+1})^{-1}(V) = \emptyset$. Take $z \in (r^M)^{-1}(V) \cap (N_{\alpha_1} \times \dots \times N_{\alpha_M})=:O$. By the Urysohn's lemma there exists $h \in C_c(O,[0,1])$ such that $h(z)=1$. Let $y:=(h^{\mathrm{Ind}_{\alpha_1,\dots,\alpha_M}^{\beta_1,\dots,\beta_M}})_{\beta_1,\dots,\beta_M \in \Lambda}$. Now we take $a:=\psi_n(\widetilde{\xi})\psi_M(y), b:=\psi_n(\widetilde{\eta})\psi_M(y)$, and $f:=(\Vert x_0 \Vert / \Vert \langle y,g \cdot y\rangle_{C_0(E^0)} \Vert)\langle y,g \cdot y\rangle_{C_0(E^0)}$. If $n_i \neq m_i$, suppose without loss of generality that $n_i >m_i$, for $\zeta_1, \dots, \zeta_{M+1} \in C_c(E,\mathbf{N},\mathbf{S}), e \in E^1$, and $(e_1,\dots,e_M) \in E^M$ with $s^M(e_M)=r(e)$. Since $(r^{M+1})^{-1}(U)=\emptyset$, it is not possible that $(e_1,\dots,e_M) \in O$. So $\psi_M(y)^* \psi_M(\gamma_1 \diamond\dots\diamond \gamma_M)\psi(\gamma_{M+1})=0$. 
Hence if $n_i \neq m_i$ then $a^* \psi_{n_i}(\xi_i)\psi_{m_i}(\eta_i)^* b=0$. Therefore
\begin{align*}
\Vert a^*xb-\pi(f) \Vert&=\Vert \psi_M(y)^*\psi_n(\widetilde{\xi})^*x_0\psi_n(\widetilde{\eta})\psi_M(y)-\pi(f) \Vert=\Vert \langle y,g \cdot y \rangle_{C_0(E^0)}-f \Vert<\epsilon.
\end{align*}

Subcase $3.2$: Suppose that $r^{-1}(V)=\emptyset$. The argument for this subcase is similar to Subcase 3.1.

Subcase $3.3$: Suppose that $(r^{n+1})^{-1}(V) \neq \emptyset$. Since $E$ is topologically free, by \cite[Lemmas~5.9, 5.6]{Katsura:TAMS04} there exist $m \geq n+1$ and a nonempty non-returning precompact open $s^m$-section $O \subset (r^m)^{-1}(V) \cap (N_{\alpha_1} \times\dots\times N_{\alpha_m})$. Fix $z \in O$. By the Urysohn's lemma there exists $h \in C_c(O,[0,1])$ such that $h(z)=1$. Define $y:=(h^{\mathrm{Ind}_{\alpha_1,\dots,\alpha_M}^{\beta_1,\dots,\beta_M}})_{\beta_1,\dots,\beta_M \in \Lambda}$. Let $a:=\psi_n(\widetilde{\xi})\psi_m(y), b:=\psi_n(\widetilde{\eta})\psi_m(y)$, and $f:=(\Vert x_0 \Vert / \Vert \langle y,g \cdot y\rangle_{C_0(E^0)} \Vert)\langle y,g \cdot y\rangle_{C_0(E^0)}$. If $n_i \neq m_i$, suppose without loss of generality that $n_i >m_i$. For $\zeta_1, \dots, \zeta_{n_i-m_i}, \gamma_1,\dots,\gamma_m \in C_c(E,\mathbf{N},\mathbf{S})$ such that $[\gamma_1 \diamond\dots\diamond \gamma_m \vert \gamma_1 \diamond\dots\diamond \gamma_m] \in C_0(O)$, we have
\begin{align*}
\psi_m(y)^* &\psi_{n_i-m_i}(\zeta_1 \diamond\dots\diamond \zeta_{n_i-m_i})\psi_m(\gamma_1 \diamond\dots\diamond \gamma_m)
\\&=\psi(\langle y, \zeta_1 \diamond\dots\diamond \zeta_{n_i-m_i} \diamond \gamma_1 \diamond\dots\diamond \gamma_{m+m_i-n_i}\rangle_{C_0(E^0)} \cdot \gamma_{m+m_i-n_i+1} \diamond\dots\diamond \gamma_m). 
\end{align*}
For $(e_{m+m_i-n_i+1},\dots,e_m) \in E^{n_i-m_i}$ and $(e_1',\dots,e_m') \in E^m$ with $s(e_m')=r(e_{m+m_i-n_i+1})$, we notice that $(e_1',\dots,e_m')$, and $(e_{n_i-m_i+1}',\dots,e_m', e_{m+m_i-n_i+1},\dots,e_m)$ can not lie in $O$ at the same time because $O$ is non-returning. So for $n_i \neq m_i$ we have $a^* \psi_{n_i}(\xi_i)\psi_{m_i}(\eta_i)^* b=0$. Therefore
\[
\Vert a^*xb-\pi(f) \Vert=\Vert \psi_m(y)^*\pi(g)\psi_m(y)-\pi(f)\Vert=\Vert \langle y,g\cdot y\rangle_{C_0(E^0)}-f \Vert<\epsilon.   \qedhere
\]
\end{proof}

\section{The Gauge-invariant Ideal Structure}

In this section we investigate the gauge-invariant ideal structure of the twisted topological graph algebra of a topological graph. We adopt the approach developed by Katsura in \cite{Katsura:ETDS06}.

Throughout the section, we fix a topological graph $E$, a cover $\mathbf{N}=\{N_\alpha\}_{\alpha \in \Lambda}$ of $E^1$ by precompact open $s$-sections, and a $1$-cocycle $\mathbf{S}=\{s_{\alpha\beta}\}_{\alpha, \beta \in \Lambda}$ relative to $\mathbf{N}$. Let $(\psi,\pi)$ be the injective universal covariant Toeplitz representation of $X(E,\mathbf{N},\mathbf{S})$ in the twisted topological graph algebra $\mathcal{O}(E,\mathbf{N},\mathbf{S})$. Let $\gamma:\mathbf{T} \to \mathrm{Aut}(\mathcal{O}(E,\mathbf{N},\mathbf{S}))$ be the gauge action. Then 
\[
B_{[0,\infty]}=\mathcal{O}(E,\mathbf{N},\mathbf{S})^\gamma:=\{a \in \mathcal{O}(E,\mathbf{N},\mathbf{S}): \gamma_z(a)=a, \text{ for all } z \in \mathbf{T}\}.
\]
Let $\Gamma:\mathcal{O}(E,\mathbf{N},\mathbf{S}) \to \mathcal{O}(E,\mathbf{N},\mathbf{S})^\gamma$ be the expectation induced from the gauge action.

\begin{defn}[{\cite[Definitions~2.1, 2.3]{Katsura:ETDS06}}]
Let $F^0$ be a closed subset of $E^0$, and let $F^1:=s^{-1}(F^0)$. Then $F^0$ is called \emph{invariant} if the quadruple $F:=(F^0,F^1,r \vert_{F^1},s \vert_{F^1})$ is a topological graph, and for $v \in E_{\mathrm{rg}}^0 \cap F^0$, we have $r^{-1}(v) \cap F^1 \neq \emptyset$. A pair $\rho=(F^0,Z)$ is an \emph{admissible pair} if $F^0$ is a closed invariant subset of $E^0, Z$ is closed in $E^0,$ and $F_{\mathrm{sg}}^0 \subset Z \subset E_{\mathrm{sg}}^0 \cap F^0$.
\end{defn}

Firstly, we aim to define a map from the set of all admissible pairs of $E$ to the set of all closed two-sided ideals of $\mathcal{O}(E,\mathbf{N},\mathbf{S})$. 

The following definition is a generalization of \cite[Definition~3.1]{Katsura:ETDS06}.

\begin{defn}\label{define J_rho I(rho)}
Let $\rho=(F^0,Z)$ be an admissible pair and let $\omega$ be the homomorphism of Proposition~\ref{omega is a homomorphism} from $\mathcal{L}(X(E,\mathbf{N},\mathbf{S}))$ to $\mathcal{L}(X(F^1,\mathbf{N}^{F^1},\mathbf{S}^{F^1}))$. Define
\begin{align*}
J_\rho:=\{\pi(f)+\psi^{(1)}(K): f(Z)=0, \mathrm{and} \ \omega(\phi(f)+K)=0 \},
\end{align*}
and define $I(\rho)$ to be the closed two-sided ideal in $\mathcal{O}(E,\mathbf{N},\mathbf{S})$ generated by $J_\rho$.
\end{defn}

The following proposition is a generalization of \cite[Proposition~3.5]{Katsura:ETDS06}.

\begin{prop}\label{simplification of I_rho}
Let $\rho=(F^0,Z)$ be an admissible pair. Then $I(\rho)$ is gauge-invariant, and 
\[
I(\rho)=\overline{\lsp}\{\psi_n(\xi)a \psi_m(\eta)^*: \xi \in C_c(E_n,\mathbf{N}^n,\mathbf{S}^n), \eta \in C_c(E_m,\mathbf{N}^m,\mathbf{S}^m),a \in J_\rho \}.
\]
\end{prop}
\begin{proof}
Since the core of $\mathcal{O}(E,\mathbf{N},\mathbf{S})$ coincides with the fixed point algebra $\mathcal{O}(E,\mathbf{N},\mathbf{S})^\gamma, \gamma$ fixes $J_\rho$. So $I(\rho)$ is gauge-invariant.

The set inclusion $\supset$ is obvious. We prove the other direction. For $\pi(f)+\psi^{(1)}(K) \in J_\rho$ and $g \in C_0(E^0)$, we have $(\pi(f)+\psi^{(1)}(K))\pi(g)=\pi(fg)+\psi^{(1)}(K\phi(g)), (fg)(Z)=0$, and $\omega(\phi(fg)+K\phi(g))=\omega(\phi(f)+K)\omega(\phi(g))=0$. For $\pi(f)+\psi^{(1)}(K) \in J_\rho$ and $x \in C_c(E,\mathbf{N},\mathbf{S})$, we have $(\pi(f)+\psi^{(1)}(K))\psi(x)=\psi((\phi(f)+K)x)$. By Proposition~\ref{omega is a homomorphism}, $(\phi(f)+K)x \in X(E,\mathbf{N},\mathbf{S})_{C_0(E^0 \setminus F^0)}$. The Cohen factorization theorem shows that $(\phi(f)+K)x=y \cdot g$ for some $y \in X(E,\mathbf{N},\mathbf{S}), g \in C_0(E^0 \setminus F^0)$. We have $g(Z)=0$ because $Z \subset F^0$. By Proposition~\ref{omega is a homomorphism} $\omega(g)=0$ since $F^0$ is invariant and $\langle g \cdot z, g \cdot z \rangle_{C_0(E^0)}(F^0)=0$ for all $z \in C_c(E,\mathbf{N},\mathbf{S})$. So $\pi(g) \in J_\rho$. Inductively, we deduce that for $a \in J_\rho$, and for $\xi \in C_c(E_m,\mathbf{N}^m,\mathbf{S}^m), a\psi_m(\xi)\in\lsp\{\psi_m(\eta)b:\eta \in C_c(E_m,\mathbf{N}^m,\mathbf{S}^m), b \in J_\rho\}$. A symmetric argument gives $\psi_m(\xi)^*a\in\lsp\{b\psi_m(\eta)^*:\eta \in C_c(E_m,\mathbf{N}^m,\mathbf{S}^m),b \in J_\rho\}$. A straightforward approximation argument then yields the required result.
\end{proof}

Secondly, we want to construct a map from the set of all closed two-sided ideals of $\mathcal{O}(E,\mathbf{N},\mathbf{S})$ to the set of all admissible pairs of $E$.

The following definition is a generalization of \cite[Definition~2.4]{Katsura:ETDS06}.

\begin{defn}\label{defined rho(I)}
Let $I$ be a closed two-sided ideal of $\mathcal{O}(E,\mathbf{N},\mathbf{S})$. Let $F_I^0, Z_I$ be the closed subsets of $E^0$ such that $\pi^{-1}(I)=C_0(E^0 \setminus F_I^0)$ and $\pi^{-1}(I+B_1)=C_0(E^0 \setminus Z_I)$. Define $\rho(I):=(F_I^0,Z_I)$, and define $F_I^1:=s^{-1}(F_I^0)$.
\end{defn}

The following lemma is a generalization of \cite[Lemma~2.6]{Katsura:ETDS06}.

\begin{lemma}\label{criteria to check whether element in I}
Let $I$ be a closed two-sided ideal of $\mathcal{O}(E,\mathbf{N},\mathbf{S})$. For $x \in X(E,\mathbf{N},\mathbf{S})$, we have $\psi(x) \in I$ if and only if $x \in X(E,\mathbf{N},\mathbf{S})_{C_0(E^0 \setminus F_I^0)}$. For $K \in \mathcal{K}(X(E,\mathbf{N},\mathbf{S})), \psi^{(1)}(K) \in I$ if and only if $\psi(Kx) \in I$ for all $x \in X(E,\mathbf{N},\mathbf{S})$ if and only if $Kx \in X(E,\mathbf{N},\mathbf{S})_{C_0(E^0 \setminus F_I^0)}$ for all $x \in X(E,\mathbf{N},\mathbf{S})$.
\end{lemma}
\begin{proof}
Fix $x \in X(E,\mathbf{N},\mathbf{S})$. Then $\psi(x) \in I$ if and only if $\pi(\langle x, x \rangle_{C_0(E^0)}) \in I$ if and only if $x \in X(E,\mathbf{N},\mathbf{S})_{C_0(E^0 \setminus F_I^0)}$.

Fix $K \in \mathcal{K}(X(E,\mathbf{N},\mathbf{S}))$. Suppose that $\psi^{(1)}(K) \in I$. For $x \in X(E,\mathbf{N},\mathbf{S})$, we have $\psi^{(1)}(K)\psi(x)=\psi(Kx) \in I$. Now suppose that $\psi(Kx) \in I$ for all $x \in X(E,\mathbf{N},\mathbf{S})$. For $x,y \in X(E,\mathbf{N},\mathbf{S})$, we have $\psi^{(1)}(K \Theta_{x,y})=\psi^{(1)}(\Theta_{Kx,y})=\psi(Kx)\psi(y)^* \in I$. So $\psi^{(1)}(K) \in I$. By the first statement, $\psi^{(1)}(K) \in I$ if and only if $Kx \in X(E,\mathbf{N},\mathbf{S})_{C_0(E^0 \setminus F_I^0)}$ for all $x \in X(E,\mathbf{N},\mathbf{S})$.
\end{proof}

The following proposition is a generalization of \cite[Proposition~2.8]{Katsura:ETDS06}.

\begin{prop}\label{rho I is admissible pair}
Let $I$ be a closed two-sided ideal of $\mathcal{O}(E,\mathbf{N},\mathbf{S})$. Then $\rho(I)$ is an admissible pair.
\end{prop}
\begin{proof}
Firstly, we prove that $F_I^0$ is invariant. Fix $e \in F_I^1 \cap N_{\alpha_0}$. Suppose that $r(e) \notin F_I^0$, for a contradiction. By the Urysohn's lemma, there exist $x \in C_0(N_{\alpha_0})$ and $f \in C_0(E^0 \setminus F_I^0)$ such that $x(e)=1$ and $f(r(e))=1$. Then $\pi(f)\psi((x^{\Ind_{\alpha_0}^{\alpha}}))=\psi(f\cdot (x^{\Ind_{\alpha_0}^{\alpha}})) \in I$ because $\pi(f) \in I$. By Lemma~\ref{criteria to check whether element in I}, $\langle f \cdot (x^{\Ind_{\alpha_0}^{\alpha}}), f \cdot (x^{\Ind_{\alpha_0}^{\alpha}}) \rangle_{C_0(E^0)} \in C_0(E^0 \setminus F^0)$. However,
\[
\langle f \cdot (x^{\Ind_{\alpha_0}^{\alpha}}), f \cdot (x^{\Ind_{\alpha_0}^{\alpha}}) \rangle_{C_0(E^0)}(s(e)) \geq [f \cdot (x^{\Ind_{\alpha_0}^{\alpha}}) \vert f \cdot (x^{\Ind_{\alpha_0}^{\alpha}})](e)=1,
\]
which is a contradiction. So $r(e) \in F_I^0$. Now fix $v \in E_{\mathrm{rg}}^0 \cap F_I^0$. Suppose that $r^{-1}(v) \cap F_I^1=\emptyset$, for a contradiction. By \cite[Lemma~1.4]{Katsura:ETDS06}, there exists an open neighborhood $V \subset E_{\mathrm{rg}}^0$ of $v$, such that $r^{-1}(V) \cap F_I^1=\emptyset$. By the Urysohn's lemma, there exists $f \in C_0(V) \subset C_0(E_{\mathrm{rg}}^0)$ such that $f(v)=1$. Then $\pi(f) \notin I$ because $v \in F_I^0$. However, since $r^{-1}(V) \cap F_I^1=\emptyset$, we have $\langle f\cdot x , f\cdot x \rangle_{C_0(E^0)}(F_I^0)=0$ for all $x \in C_c(E,\mathbf{N},\mathbf{S})$. By Lemma~\ref{criteria to check whether element in I}, $\psi^{(1)}(\phi(f)) \in I$. By the covariance of $(\psi,\pi)$, we get $\pi(f)=\psi^{(1)}(f)$, which is a contradiction. So $r^{-1}(v) \cap F_I^1\neq\emptyset$, and $F_I^0$ is invariant.

Now we show that $Z_I \subset E_{\mathrm{sg}}^0 \cap F_I^0$. It is obvious that $Z_I \subset F_I^0$. Fix $v \in Z_I$. Suppose that $v \in E_{\mathrm{rg}}^0$, for a contradiction. By the Urysohn's lemma, there exists $f \in C_0(E_{\mathrm{rg}}^0)$ such that $f(v)=1$. By the covariance of $(\psi,\pi)$, we have $\pi(f)=\psi^{(1)}(\phi(f)) \in I+B_1$. So $f(Z_I)=0$, which is a contradiction. Hence $v \in E_{\mathrm{sg}}^0$ and $Z_I \subset E_{\mathrm{sg}}^0 \cap F_I^0$.

Finally we prove that $(F_I^0)_{\mathrm{sg}} \subset Z_I$. It is equivalent to show that $F_I^0 \setminus Z_I \subset (F_I^0)_{\mathrm{rg}}$. Fix $v \in F_I^0 \setminus Z_I$. Suppose that $v \in \overline{F_I^0 \setminus \overline{r(F_I^1)}}$ for a contradiction. Choose an open neighborhood $V$ of $v$ such that $V \cap Z_I=\emptyset$. Then there exists $w \in V \setminus \overline{r(F_I^1)}$. By the Urysohn's lemma, there exists $f \in C_0(E^0)$ such that $f(w)=1,f(\overline{r(F_I^1)})=0$, and $f(Z_I)=0$. So $\pi(f)=i+\psi^{(1)}(K)\in I+B_1$, and $f \cdot x \in X(E,\mathbf{N},\mathbf{S})_{C_0(E^0 \setminus F_I^0)}$ for all $x \in X(E,\mathbf{N},\mathbf{S})$. By Lemma~\ref{criteria to check whether element in I}, $\psi(f \cdot x)=\pi(f)\psi(x)=i\psi(x)+\psi(Kx) \in I$ for all $x \in X(E,\mathbf{N},\mathbf{S})$. By Lemma~\ref{criteria to check whether element in I} again, $\psi^{(1)}(K) \in I$. So $\pi(f) \in I$ and $f(F_I^0)=0$, which contradicts with $f(w)=1$. Hence $v \in F_I^0 \setminus \overline{F_I^0 \setminus \overline{r(F_I^1)}}$.

Suppose that $v \notin (F_I^0)_{\mathrm{fin}}$, for a contradiction. Choose a compact neighborhood $C \subset F_I^0$ of $v$ such that $C \cap Z_I=\emptyset$. By the Urysohn's lemma, there exists $f \in C_0(E^0)$ such that $f(Z_I)=0$ and $f(C)=1$. Then $\pi(f)=i+\psi^{(1)}(K) \in I+B_1$. By Lemma~\ref{criteria to check whether element in I}, $f\cdot x-Kx \in X(E,\mathbf{N},\mathbf{S})_{C_0(E^0\setminus F_I^0)}$ for all $x \in X(E,\mathbf{N},\mathbf{S})$. Take $\{ x_i,y_i \}_{i=1}^{n} \subset C_c(E,\mathbf{N},\mathbf{S})$ with $\Vert \sum_{i=1}^{n}\Theta_{x_i,y_i}-K \Vert <1/2$. Since $r^{-1}(C)$ is not compact, there exists $e \in (r^{-1}(C) \cap F_I^1 \cap N_{\alpha_0}) \setminus \bigcup_{i=1}^{n}(\supp([x_i \vert x_i]) \cup \supp([y_i \vert y_i]))$. The Urysohn's lemma gives $g\in C_0(N_{\alpha_0})$ such that $g\big(\bigcup_{i=1}^{n}(\supp([x_i \vert x_i]) \cup \supp([y_i \vert y_i]))\big)=0$ and $g(e)=1$. Let $x:=(g^{\mathrm{Ind}_{\alpha_0}^{\alpha}})$. For any $w \in F_I^0$, we have
\begin{align*}
\Big\vert \Big\langle f \cdot x-\sum_{i=1}^{n}\Theta_{x_i,y_i}x,x \Big\rangle_{C_0(E^0)}(w) \Big\vert&\leq \Big\vert \Big\langle Kx-\sum_{i=1}^{n}\Theta_{x_i,y_i}x,x \Big\rangle_{C_0(E^0)}(w) \Big\vert
\\& \leq   \Big\Vert K-\sum_{i=1}^{n}\Theta_{x_i,y_i} \Big\Vert<1/2.
\end{align*}
However,
\begin{align*}
\Big\vert \Big\langle f \cdot x-\sum_{i=1}^{n}\Theta_{x_i,y_i}x,x \Big\rangle_{C_0(E^0)}(s(e)) \Big\vert &=\vert [f \cdot x \vert x ](e)\vert=1,
\end{align*}
which is a contradiction. So $v \in (F_I^0)_{\mathrm{fin}}$ and $v \in (F_I^0)_{\mathrm{rg}}$. Hence $\rho(I)$ is an admissible pair.
\end{proof}

We have now defined a map from the set of all admissible pairs of $E$ to the set of all closed two-sided ideals of $\mathcal{O}(E,\mathbf{N},\mathbf{S})$ by $\rho \to I(\rho)$, and defined a map from the set of all closed two-sided ideals of $\mathcal{O}(E,\mathbf{N},\mathbf{S})$ to the set of all admissible pairs of $E$ by $I \to \rho(I)$. Unfortunately, these two maps are not inverse to each other in general: $\rho \to I(\rho)$ is not surjective and $I \to \rho(I)$ is not injective. However, these two maps are in fact inverse to each other when we restrict to the set of all gauge-invariant closed two-sided ideals of $\mathcal{O}(E,\mathbf{N},\mathbf{S})$ (see Theorem~\ref{one-to-one corr between adm pairs and gauge-inv ideals}).

The following theorem is a generalization of \cite[Proposition~3.10]{Katsura:ETDS06}.

\begin{thm}\label{rho I_rho=rho}
Let $\rho=(F^0,Z)$ be an admissible pair of $E$. Then $\rho(I(\rho))=\rho$. 
\end{thm}
\begin{proof}
To prove that $Z=Z_{I(\rho)}$, it suffices to show that $\pi^{-1}(I(\rho)+B_1)=C_0(E^0 \setminus Z)$. Fix $f \in \pi^{-1}(I(\rho)+B_1)$, and fix $v \in Z$. Then $\pi(f)=i+\psi^{(1)}(K) \in I(\rho)+B_1$. For $\epsilon>0$, by Proposition~\ref{simplification of I_rho}, there exist $\xi_i \in C_c(E_{n_i},\mathbf{N}^{n_i},\mathbf{S}^{n_i}), \eta_i \in C_c(E_{m_i},\mathbf{N}^{m_i},\mathbf{S}^{m_i})$, and $\pi(f_i)+\psi^{(1)}(K_i) \in J_\rho$, such that
\[
\Big\Vert \sum_i \psi_{n_i}(\xi_i)(\pi(f_i)+\psi^{(1)}(K_i))\psi_{m_i}(\eta_i)^*-(\pi(f)-\psi^{(1)}(K)) \Big\Vert<\epsilon.
\]
By Proposition~\ref{define pi_{k}^{n} when 0 leq k <n}, we have
\begin{align*}
\Big\Vert \pi_0^\infty\circ\Gamma\Big( \sum_i \psi_{n_i}(\xi_i)(\pi(f_i)&+\psi^{(1)}(K_i))\psi_{m_i}(\eta_i)^*-(\pi(f)-\psi^{(1)}(K))\Big) \Big\Vert\\&=\Big\Vert \sum_{n_i=m_i=0}(\xi_if_i\eta_i^*)\vert_{E_{\mathrm{sg}}^0}-f \vert_{E_{\mathrm{sg}}^0} \Big\Vert<\epsilon.
\end{align*}
So $\vert \sum_{n_i=m_i=0}(\xi_if_i\eta_i^*)(v)-f(v) \vert=\vert f(v) \vert<\epsilon$, giving $f \in C_0(E^0 \setminus Z)$. Conversely, fix $f \in C_0(E^0 \setminus Z)$. Let $\omega:\mathcal{L}(X(E,\mathbf{N},\mathbf{S})) \to \mathcal{L}(X(F^1,\mathbf{N}^{F^1},\mathbf{S}^{F^1}))$ be the homomorphism of Proposition~\ref{omega is a homomorphism}. Since $F_{\mathrm{sg}}^0 \subset Z$, we have $\phi(f \vert_{F^0}) \in \mathcal{K}(X(F^1,\mathbf{N}^{F^1},\mathbf{S}^{F^1}))$. By Proposition~\ref{omega maps K(X(E,N,S)) onto K(X(X^1,N^X^1,S^X^1))}, there exists $K \in \mathcal{K}(X(E,\mathbf{N},\mathbf{S}))$ such that $\omega(K)=\phi(f \vert_{F^0})$. Since $\omega(\phi(f)-K)=0$, we have $\pi(f)+\psi^{(1)}(K) \in J_\rho$. So $f \in \pi^{-1}(I(\rho)+B_1)$.

We verify that $F_{I(\rho)}^0=F^0$. Fix $f \in C_0(E^0 \setminus F^0)$. Since $Z \subset F^0$, we have $f(Z)=0$. Since $\phi(f) x \in X(E,\mathbf{N},\mathbf{S})_{C_0(E^0 \setminus F^0)}$ for all $x \in X(E,\mathbf{N},\mathbf{S})$, we have $\phi(f) \in \ker(\omega)$. So $\pi(f) \in J_\rho \subset I(\rho)$. Hence $C_0(E^0 \setminus F^0) \subset \pi^{-1}(I(\rho))$ and $F_{I(\rho)}^0 \subset F^0$. Fix $v \in F^0$. Suppose that $v \notin F_{I(\rho)}^0$, for a contradiction. By the Urysohn's lemma, there exists $f \in C_0(E^0 \setminus F_{I(\rho)}^0)$ such that $f(v)=1$. So $\pi(f) \in I_\rho$. By Proposition~\ref{rho I is admissible pair}, $Z=Z_{I(\rho)} \subset F_{I(\rho)}^0$. We consider two cases.

Case $1$: There exist $n \geq 1$ and $(e_1,\dots,e_n) \in (N_{\alpha_1} \times\dots\times N_{\alpha_n}) \cap E^n$ such that $r(e_1)=v$ and $s(e_n) \in Z$. By the Urysohn's lemma there exists $g \in C_0(N_{\alpha_1} \times\dots\times N_{\alpha_n} \cap E^n)$ such that $g(e_1,\dots,e_n)=1$. Let $\xi:=(g^{\Ind_{\alpha_1, \dots, \alpha_n}^{\beta_1,\dots,\beta_n}})$. Since $\pi(f) \in I(\rho)$, we have $\pi(\langle f \cdot \xi,f \cdot \xi \rangle_{C_0(E^0)}) \in I(\rho)$. So $\langle f \cdot \xi,f \cdot \xi \rangle_{C_0(E^0)}(F_{I(\rho)}^0)=0$. However, 
\begin{align*}
\langle f \cdot \xi,f \cdot \xi \rangle_{C_0(E^0)}(s(e_n))&\geq [f \cdot \xi,f \cdot \xi](e_1,\dots,e_n)
=\vert f(r^n(e_1,\dots,e_n))g(e_1,\dots,e_n)\vert^2=1,
\end{align*}
which contradicts $\langle f \cdot \xi,f \cdot \xi \rangle_{C_0(E^0)}(F_{I(\rho)}^0)=0$.

Case $2$: For any $n \geq 1$ there exists $(e_1,\dots,e_n) \in E^n$ such that $r(e_1)=v$ and $s(e_n) \in F^0$. Since $\pi(f) \in I(\rho)$, by Proposition~\ref{simplification of I_rho}, there exist $\xi_i \in C_c(E_{n_i},\mathbf{N}^{n_i},\mathbf{S}^{n_i}), \eta_i \in C_c(E_{m_i},\mathbf{N}^{m_i},\mathbf{S}^{m_i})$, and $\pi(f_i)+\psi^{(1)}(K_i) \in J_\rho$, such that
\[
\Big\Vert \sum_i \psi_{n_i}(\xi_i)(\pi(f_i)+\psi^{(1)}(K_i))\psi_{m_i}(\eta_i)^*-\pi(f) \Big\Vert<1/2.
\]
So
\begin{align*}
\Big\Vert \Gamma\Big(\sum_i \psi_{n_i}(\xi_i)(\pi(f_i)&+\psi^{(1)}(K_i))\psi_{m_i}(\eta_i)^*-\pi(f)\Big) \Big\Vert
\\&\leq\Big\Vert \sum_{n_i=m_i} \psi_{n_i}(\xi_i)(\pi(f_i)+\psi^{(1)}(K_i))\psi_{m_i}(\eta_i)^*-\pi(f) \Big\Vert<1/2.
\end{align*}
Let $n:=\max\{n_i:n_i=m_i\}+1$. Then there exists $(e_1,\dots,e_n) \in (N_{\alpha_1} \times\dots\times N_{\alpha_n} ) \cap E^n$ such that $r(e_1)=v$ and $s(e_n) \in F^0$. By the Urysohn's lemma, there exists $g \in C_0((N_{\alpha_1} \times\dots\times N_{\alpha_n}) \cap E^n)$, such that $g(e_1,\dots,e_n)=1$. Let $\xi:=(g^{\Ind_{\alpha_1, \dots, \alpha_n}^{\beta_1,\dots,\beta_n}})$. Since $(N_{\alpha_1} \times\dots\times N_{\alpha_n} )\cap E^n$ is an $s^n$-section, we have $\Vert \xi\Vert_{C_0(E^0)}=1$. We notice that if $n_i=m_i$ then 
\[
\psi_n(\xi)^*\psi_{n_i}(\xi_i)(\pi(f_i)+\psi^{(1)}(K_i))\psi_{m_i}(\eta_i)^*\psi_n(\xi) \in \pi(C_0(E^0 \setminus F^0)).
\]
So for any $w \in F^0, \vert \langle \xi,f \cdot \xi \rangle_{C_0(E^0)}(w)\vert<1/2$. However, 
\[
\vert \langle \xi,f \cdot \xi \rangle_{C_0(E^0)}(s(e_1,\dots,e_n))\vert=\vert f(r(e_1,\dots,e_n))g(e_1,\dots,e_n)^2\vert=1,
\]
which is a contradiction. Hence $v \in F_{I(\rho)}^0,F_{I(\rho)}^0=F^0$, and $\rho(I(\rho))=\rho$.
\end{proof}

The following proposition is a generalization of \cite[Lemma~3.11]{Katsura:ETDS06}.

\begin{prop}\label{J_rho I,I_rho I subset I}
Let $I$ be a closed two-sided ideal of $\mathcal{O}(E,\mathbf{N},\mathbf{S})$. Then $J_{\rho(I)} \subset I$, and hence $I(\rho(I)) \subset I$.
\end{prop}
\begin{proof}
Let $\omega:\mathcal{L}(X(E,\mathbf{N},\mathbf{S})) \to \mathcal{L}(X(F_I^1,\mathbf{N}^{F_I^1},\mathbf{S}^{F_I^1}))$ be the homomorphism in Proposition~\ref{omega is a homomorphism}. Fix $\pi(f)+\psi^{(1)}(K) \in J_{\rho(I)}$. Since $f(Z_I)=0, \pi(f)=i+\psi^{(1)}(K') \in I+B_1$. Since $\omega(\phi(f)+K)=0$, by Proposition~\ref{omega is a homomorphism}, $f \cdot x+Kx \in X(E,\mathbf{N},\mathbf{S})_{C_0(E^0 \setminus F_I^0)}$ for all $x \in X(E,\mathbf{N},\mathbf{S})$. By Lemma~\ref{criteria to check whether element in I}, $\pi(f)\psi(x)+\psi(Kx) \in I$. So $\psi((K+K')x) \in I$. By Lemma~\ref{criteria to check whether element in I}, $\psi^{(1)}(K)+\psi^{(1)}(K') \in I$. So $\pi(f)+\psi^{(1)}(K) \in I$. Hence $J_{\rho(I)} \subset I$. Since $I(\rho(I))$ is a closed two-sided ideal of $\mathcal{O}(E,\mathbf{N},\mathbf{S})$ generated by $J_{\rho(I)}$, the result follows.
\end{proof}

Let $I$ be a closed two-sided ideal of $\mathcal{O}(E,\mathbf{N},\mathbf{S})$. Next we prove that if $I$ is gauge-invariant then $I(\rho(I))=I$. We set up some notation. By Proposition~\ref{rho I is admissible pair}, $\rho(I)=(F_I^0,Z_I)$ is an admissible pair and $F_I:=(F_I^0,F_I^1,r \vert_{F_I^1},s \vert_{F_I^1})$ is a topological graph. Define a closed subset $Y_I:=Z_I \cap (F_I^0)_{\mathrm{rg}}$ of $(F_I^0)_{\mathrm{rg}}$ in the subspace topology of $(F_I^0)_{\mathrm{rg}}$. By Definition~\ref{define E_Y where Y is a closed subset of E_rg^0}, define a topological graph $(F_I)_{Y_I}:=((F_I^0)_{Y_I},(F_I^1)_{Y_I},r_{Y_I},s_{Y_I})$, define a cover $(\mathbf{N}^{F_I^1})_{Y_I}$ of $(F_I^1)_{Y_I}$ by precompact open $s_{Y_I}$-sections, and define a $1$-cocycle $(\mathbf{S}^{F_I^1})_{Y_I}$ relative to $(\mathbf{N}^{F_I^1})_{Y_I}$. As described in the paragraph following Definition~\ref{define E_Y where Y is a closed subset of E_rg^0}, let $p_I^0:(F_I^0)_{Y_I} \to F_I^0$ and $p_I^1:(F_I^1)_{Y_I} \to F_I^1$ be two projections. Let $(p_I^0)_*:C_0(F_I^0) \to C_0((F_I^0)_{Y_I})$ be the homomorphism obtained from $p_I^0$, and let $(p_I^1)_*:X(F_I^1,\mathbf{N}^{F_I^1},\mathbf{S}^{F_I^1}) \to X((F_I^1)_{Y_I},(\mathbf{N}^{F_I^1})_{Y_I},(\mathbf{S}^{F_I^1})_{Y_I})$ be the norm-preserving linear map obtained from $p_I^1$. Let $(t_I^0,t_I^1)$ be the injective universal covariant Toeplitz representation of $X((F_I)_{Y_I},(\mathbf{N}^{F_I^1})_{Y_I},(\mathbf{S}^{F_I^1})_{Y_I})$ in $\mathcal{O}((F_I)_{Y_I},(\mathbf{N}^{F_I^1})_{Y_I},(\mathbf{S}^{F_I^1})_{Y_I})$. 

The following lemma is a generalization of \cite[Proposition~3.15]{Katsura:ETDS06}.

\begin{lemma}\label{construct (psi_rho I,pi_rho I) from an ideal I of O(E,N,S)}
Let $I$ be a closed two-sided ideal of $\mathcal{O}(E,\mathbf{N},\mathbf{S})$. Then there exists an injective Toeplitz representation $(\psi_I,\pi_I)$ of $X(F_I^1,\mathbf{N}^{F_I^1},\mathbf{S}^{F_I^1})$ in $\mathcal{O}(E,\mathbf{N},\mathbf{S}) / I$ such that $C^*(\psi_I,\pi_I)=\mathcal{O}(E,\mathbf{N},\mathbf{S}) / I$, and 
\begin{equation}\label{Y_{psi_I}=Y_I}
C_0((F_I^0)_{\mathrm{rg}} \setminus Y_I)=\{f \in C_0((F_I^0)_{\mathrm{rg}}): \psi_I^{(1)}(\phi(f))=\pi_I(f) \}.
\end{equation}
\end{lemma}
\begin{proof}
Let $\omega:\mathcal{L}(X(E,\mathbf{N},\mathbf{S})) \to \mathcal{L}(X(F_I^1,\mathbf{N}^{F_I^1},\mathbf{S}^{F_I^1}))$ be the homomorphism in Proposition~\ref{omega is a homomorphism}. For $x \in C_c(F_I^1,\mathbf{N}^{F_I^1},\mathbf{S}^{F_I^1})$, and for $y,z \in C_c(E,\mathbf{N},\mathbf{S})$ such that $y_\alpha \vert_{\overline{N_\alpha \cap F_I^1}}=z_\alpha \vert_{\overline{N_\alpha \cap F_I^1}}=x_\alpha$ for all $\alpha \in \Lambda$. We have $\langle y-z,y-z\rangle_{C_0(E^0)} \in C_0(E^0 \setminus F_I^0)$. So $\pi(\langle y-z,y-z \rangle_{C_0(E^0)}) \in I$ and $\psi(y-z) \in I$. Hence there is a bounded linear map $\psi_I:X(F_I^1,\mathbf{N}^{F_I^1},\mathbf{S}^{F_I^1}) \to \mathcal{O}(E,\mathbf{N},\mathbf{S})/I$ such that $\psi_I(x)=\psi(y)+I$. For $f \in C_0(F_I^0)$, for extensions $g,h \in C_0(E^0)$ of $f$, we have $g-h \in C_0(E^0 \setminus F_I^0)$. So $\pi(g-h) \in I$. So there is a homomorphism $\pi_I:C_0(F_I^0) \to \mathcal{O}(E,\mathbf{N},\mathbf{S})/I$ such that $\pi_I(f)=\pi(g)+I$. It is straightforward to check that $(\psi_I,\pi_I)$ is an injective Toeplitz representation of $X(F_I^1,\mathbf{N}^{F_I^1},\mathbf{S}^{F_I^1})$ such that $C^*(\psi_I,\pi_I)=\mathcal{O}(E,\mathbf{N},\mathbf{S}) / I$.

Now we prove Equation~(\ref{Y_{psi_I}=Y_I}). We notice that $C_0((F_I^0)_{\mathrm{rg}} \setminus Y_I)=C_0(F_I^0 \setminus Z_I)$. Fix $f \in C_0(F_I^0 \setminus Z_I)$. Take an extension $\widetilde{f}$ of $f$ in $C_0(E^0)$. By definition of $Z_I$, we have $\pi(\widetilde{f})=\psi^{(1)}(K)+i$ for some $K \in \mathcal{K}(X(E,\mathbf{N},\mathbf{S}))$ and $i \in I$. Then $\pi_I(f)=\pi(\widetilde{f})+I=\psi_I^{(1)}(\omega(K))$. Since $\pi_I$ is injective, by \cite[Proposition~3.3]{Katsura:JFA04}, we have $\pi_I(f)=\psi_I^{(1)}(\phi(f))$. Conversely, fix $f \in C_0((F_I^0)_{\mathrm{rg}})$ with $\psi_I^{(1)}(\phi(f))=\pi_I(f)$. Then $\phi(f) \in \mathcal{K}(X(F_I^1,\mathbf{N}^{F_I^1},\mathbf{S}^{F_I^1}))$. By Proposition~\ref{omega maps K(X(E,N,S)) onto K(X(X^1,N^X^1,S^X^1))}, there exists $K \in \mathcal{K}(X(E,\mathbf{N},\mathbf{S}))$, such that $\omega(K)=\phi(f)$. Take an extension $\widetilde{f}$ of $f$ in $C_0(E^0)$. Then $\psi^{(1)}(K)+I=\pi(\widetilde{f})+I$. By definition of $Z_I$, we have $f(Z_I)=0$. So Equation~(\ref{Y_{psi_I}=Y_I}) holds.
\end{proof}

The following theorem generalizes \cite[Proposition~3.16]{Katsura:ETDS06}.

\begin{thm}\label{I(rho(I))=I if I is gauge-inv}
Let $I$ be a closed two-sided ideal of $\mathcal{O}(E,\mathbf{N},\mathbf{S})$ which is gauge-invariant. Then $I(\rho (I))= I$.
\end{thm}
\begin{proof}
By Proposition~\ref{J_rho I,I_rho I subset I}, $I(\rho(I)) \subset I$. So there is a well-defined quotient map $q: \mathcal{O}(E,\mathbf{N},\mathbf{S})/I(\rho(I)) \to \mathcal{O}(E,\mathbf{N},\mathbf{S})/I$. Lemma~\ref{construct (psi_rho I,pi_rho I) from an ideal I of O(E,N,S)} yields an injective Toeplitz representation $(\psi_I,\pi_I)$ of $X(F_I,\mathbf{N}^{F_I^1},\mathbf{S}^{F_I^1})$ in $\mathcal{O}(E,\mathbf{N},\mathbf{S}) / I$ such that $C^*(\psi_I,\pi_I)=\mathcal{O}(E,\mathbf{N},\mathbf{S}) / I$ and Equation~(\ref{Y_{psi_I}=Y_I}) holds. Proposition~\ref{tilde{psi},tilde{pi} is a cov Toep rep} gives an injective covariant Toeplitz representation $(\widetilde{\psi}_I,\widetilde{\pi}_I)$ of $X((F_I)_{Y_I},(\mathbf{N}^{F_I^1})_{Y_I},(\mathbf{S}^{F_I^1})_{Y_I})$ in $\mathcal{O}(E,\mathbf{N},\mathbf{S}) / I$ such that $\widetilde{\pi}_I \circ(p_I^0)_*=\pi_I, \widetilde{\psi}_I \circ(p_I^1)_*=\psi_I$, and $C^*(\widetilde{\psi}_I,\widetilde{\pi}_I)=\mathcal{O}(E,\mathbf{N},\mathbf{S}) / I$. So there is a surjective homomorphism $\varphi:\mathcal{O}((F_I)_{Y_I},(\mathbf{N}^{F_I^1})_{Y_I},(\mathbf{S}^{F_I^1})_{Y_I})\to\mathcal{O}(E,\mathbf{N},\mathbf{S}) / I$ such that $\varphi \circ t_I^0=\widetilde{\pi}_I$ and $\varphi\circ t_I^1=\widetilde{\psi}_I$. Hence $\varphi \circ t_I^0\circ (p_I^0)_*=\pi_I$ and $\varphi \circ t_I^1\circ (p_I^1)_*=\psi_I$. Since $I$ is gauge-invariant, there is a gauge action $\gamma_I:\mathbb{T} \to \mathrm{Aut}(\mathcal{O}(E,\mathbf{N},\mathbf{S}) / I)$ such that $\gamma_I(z)(a+I)=\gamma_z(a)+I$ for all $a \in \mathcal{O}(E,\mathbf{N},\mathbf{S})$. By the gauge-invariant uniqueness theorem, $\varphi$ is an isomorphism.

Let $(\psi_{I(\rho(I))},\pi_{I(\rho(I))})$ be an injective Toeplitz representation of $X(F_I,\mathbf{N}^{F_I^1},\mathbf{S}^{F_I^1})$ in the quotient $\mathcal{O}(E,\mathbf{N},\mathbf{S}) / {I(\rho(I))}$ obtained from Lemma~\ref{construct (psi_rho I,pi_rho I) from an ideal I of O(E,N,S)}. By Theorem~\ref{rho I_rho=rho}, $\rho(I(\rho(I)))=\rho(I)$. Since $I(\rho(I))$ is gauge-invariant by Proposition~\ref{simplification of I_rho}, we repeat the argument in the previous paragraph, then we obtain an isomorphism $\varphi':\mathcal{O}((F_I)_{Y_I},(\mathbf{N}^{F_I^1})_{Y_I},(\mathbf{S}^{F_I^1})_{Y_I})\to\mathcal{O}(E,\mathbf{N},\mathbf{S}) / I(\rho(I))$ such that $\varphi' \circ t_I^0\circ (p_I^0)_*=\pi_{I(\rho(I))}$ and $\varphi' \circ t_I^1\circ (p_I^1)_*=\psi_{I(\rho(I))}$. Hence $\varphi\circ\varphi'^{-1}$ is an isomorphism.

For $x \in C_c(E,\mathbf{N},\mathbf{S})$, we have
\begin{align*}
\varphi\circ\varphi'^{-1}(\psi(x)+I(\rho(I)))&=\varphi\circ\varphi'^{-1} \circ \psi_{I(\rho(I))} (x_\alpha \vert_{\overline{N_\alpha \cap F_I^1}})=\psi_I(x_\alpha \vert_{\overline{N_\alpha \cap F_I^1}})=\psi(x)+I.
\end{align*}
For $f \in C_0(E^0)$, we have
\begin{align*}
\varphi\circ\varphi'^{-1}(\pi(f)+I(\rho(I)))&=\varphi\circ\varphi'^{-1}\circ \pi_{I(\rho(I))}(f \vert_{F_I^0})=\pi_I(f \vert_{F_I^0})=\pi(f)+I.
\end{align*}
So $\varphi\circ\varphi'^{-1}=q$. Hence $I(\rho (I))= I$.
\end{proof}

\begin{thm}\label{one-to-one corr between adm pairs and gauge-inv ideals}
The map $\rho \to I(\rho)$ from the set of all admissible pairs of $E$ to the set of all gauge-invariant closed two-sided ideals of $\mathcal{O}(E,\mathbf{N},\mathbf{S})$ is a bijection with inverse $I \to \rho(I)$.
\end{thm}
\begin{proof}
This is a direct consequence of Theorem~\ref{rho I_rho=rho} and Theorem~\ref{I(rho(I))=I if I is gauge-inv}.
\end{proof}

\section{Simplicity Conditions}

In our final section we show that the twisted topological graph algebra of a topological graph is simple  if and only if the ordinary topological graph algebra is simple. Our result generalizes \cite[Theorem~8.12]{Katsura:ETDS06}.

\begin{thm}\label{equ conditions of simple algebras}
Let $E$ be a topological graph, let $\mathbf{N}=\{N_\alpha\}_{\alpha \in \Lambda}$ be a cover of $E^1$ by precompact open $s$-sections, and let $\mathbf{S}=\{s_{\alpha\beta}\}_{\alpha,\beta \in \Lambda}$ be a $1$-cocycle relative to $\mathbf{N}$. Then the following conditions are equivalent:
\begin{enumerate}
\item\label{O(E,N,S) is simple} $\mathcal{O}(E,\mathbf{N},\mathbf{S})$ is simple;
\item\label{E is minimal and top free} $E$ is minimal (see \cite[Definition~8.8]{Katsura:ETDS06}) and topologically free;
\item\label{E is minimal and not gen by a cycle} $E$ is minimal and not generated by a cycle (see \cite[Definition~8.4]{Katsura:ETDS06});
\item\label{E is minimal and free} $E$ is minimal and free (see \cite[Definition~7.2]{Katsura:ETDS06});
\item\label{O(E) is simple} $\mathcal{O}(E)$ is simple.
\end{enumerate}
\end{thm}
\begin{proof}
Let $(\psi,\pi)$ be the injective universal covariant Toeplitz representation of $X(E,\mathbf{N},\mathbf{S})$ in $\mathcal{O}(E,\mathbf{N},\mathbf{S})$.

Firstly we prove $(\ref{E is minimal and top free}) \implies (\ref{O(E,N,S) is simple})$. Suppose that $E$ is minimal and topologically free. Fix a closed two-sided ideal $L$ of $\mathcal{O}(E,\mathbf{N},\mathbf{S})$ such that $L \neq \mathcal{O}(E,\mathbf{N},\mathbf{S})$. The minimality of $E$ implies that there are only two admissible pairs $(\emptyset,\emptyset)$ and $(E^0,E_{\mathrm{sg}}^0)$. By Theorem~\ref{one-to-one corr between adm pairs and gauge-inv ideals}, there is a bijection from the set of all admissible pairs of $E$ onto the set of all gauge-invariant closed two-sided ideal of $\mathcal{O}(E,\mathbf{N},\mathbf{S})$ such that $I(\emptyset,\emptyset)=\mathcal{O}(E,\mathbf{N},\mathbf{S})$ and $I(E^0,E_{\mathrm{sg}}^0)=\{0\}$. Since $I(\rho(L)) \subset L$ by Proposition~\ref{J_rho I,I_rho I subset I}, we deduce that $I(\rho(L))=\{0\}$. So $\rho(L)=(E^0,E_{\mathrm{sg}}^0)=:(F_L^0,Z_L)$. Define $\psi_L:X(E,\mathbf{N},\mathbf{S}) \to \mathcal{O}(E,\mathbf{N},\mathbf{S})/L$ by $\psi_L(x):=\psi(x)+L$, and define $\pi_L:C_0(E^0) \to \mathcal{O}(E,\mathbf{N},\mathbf{S})/L$ by $\pi_L(f):=\pi(f)+L$. Then $(\psi_L,\pi_L)$ is a covariant Toeplitz representation of $X(E,\mathbf{N},\mathbf{S})$. By Definition~\ref{defined rho(I)}, we have $\pi_L$ is injective. By the universal property of $(\psi,\pi)$, there is a homomorphism $h:\mathcal{O}(E,\mathbf{N},\mathbf{S}) \to \mathcal{O}(E,\mathbf{N},\mathbf{S})/L$ such that $h \circ \psi=\psi_L$ and $h \circ \pi=\pi_L$. So $h$ coincides with the quotient map $q:\mathcal{O}(E,\mathbf{N},\mathbf{S}) \to \mathcal{O}(E,\mathbf{N},\mathbf{S})/I$. Since $E$ is topologically free, by Theorem~\ref{twisted version of Cuntz-Krieger Uniqueness Theorem} (the Cuntz-Krieger uniqueness theorem), the quotient map $q$ is an isomorphism. Hence $L=\{0\}$. Therefore $\mathcal{O}(E,\mathbf{N},\mathbf{S})$ is simple.

Secondly we prove $(\ref{O(E,N,S) is simple}) \implies (\ref{E is minimal and not gen by a cycle})$. Suppose that $\mathcal{O}(E,\mathbf{N},\mathbf{S})$ is simple. By Theorem~\ref{one-to-one corr between adm pairs and gauge-inv ideals}, there are only two admissible pairs $(\emptyset,\emptyset)$ and $(E^0,E_{\mathrm{sg}}^0)$. So $E$ is minimal. Suppose that $E$ is generated by a cycle, for a contradiction. Then $E^0$ is discrete by \cite[Definition~8.4]{Katsura:ETDS06}. So $E^1$ is also discrete and $H^1(E^1,\mathcal{S})=\{0\}$. By \cite[Theorem~3.3.3]{HLiPhDThesis}, $X(E,\mathbf{N},\mathbf{S}) \cong X(E)$. So $\mathcal{O}(E,\mathbf{N},\mathbf{S}) \cong \mathcal{O}(E)$. Hence $\mathcal{O}(E)$ is simple. By \cite[Theorem~8.12]{Katsura:ETDS06} $E$ is not generated by a cycle, which is a contradiction. Therefore $E$ is not generated by a cycle.

The implication $(\ref{E is minimal and not gen by a cycle}) \implies (\ref{E is minimal and top free})$ follows from \cite[Theorem~8.12]{Katsura:ETDS06}. So $(\ref{O(E,N,S) is simple}) \iff (\ref{E is minimal and top free}) \iff (\ref{E is minimal and not gen by a cycle})$. Again by \cite[Theorem~8.12]{Katsura:ETDS06}, we have $(\ref{E is minimal and not gen by a cycle}) \iff (\ref{E is minimal and free}) \iff (\ref{O(E) is simple})$.
\end{proof}

\begin{rmk}
Theorem~\ref{equ conditions of simple algebras} tells us that the twisted topological graph algebra $\mathcal{O}(E,\mathbf{N},\mathbf{S})$ is simple if and only if $\mathcal{O}(E)$ is simple. So our $1$-cocycle twisting data does not affect the simplicity of the original topological graph algebra at all.
\end{rmk}

\section*{Acknowledgments}
The author would like to thank the University of Wollongong and his PhD supervisors Professor David Pask and Professor Aidan Sims for supporting this research. The author in particular wants to thank Professor David Pask and Professor Aidan Sims for spending much time to share their ideas and taking lots of effort to read many versions of the author's PhD thesis and drafts of this paper.

\end{document}